\RequirePackage{fix-cm}
\documentclass{svjour3}                     % onecolumn (standard format)
\smartqed  % flush right qed marks, e.g. at end of proof
\usepackage{graphicx,epstopdf}

\usepackage{latexsym,enumerate}
\usepackage{mathrsfs}
\usepackage{amsfonts}
\usepackage{amssymb,amsmath}
\usepackage{epsfig}       %for Postscript files
\usepackage{subfigure}    %for parallel figures
\usepackage{graphicx}
\usepackage{float}
\usepackage{multirow}
\usepackage{array}
\usepackage{color}

\usepackage{listings}
\usepackage[ruled]{algorithm2e}

\usepackage{lipsum}
\usepackage{amsfonts}
\usepackage{graphicx}
\usepackage{epstopdf}
\usepackage{algorithmic}

\usepackage{multirow}
\usepackage{latexsym,enumerate}
\usepackage{amssymb}
\usepackage{subfigure}
\usepackage{booktabs}
\usepackage{listings}
\numberwithin{lemma}{section}
\numberwithin{theorem}{section}
\numberwithin{corollary}{section}
\numberwithin{definition}{section}
\numberwithin{property}{section}
\numberwithin{remark}{section}
\numberwithin{example}{section}
\numberwithin{figure}{section}
\begin{document}

\title{Smoothing fast iterative hard thresholding algorithm for $\ell_0$ regularized nonsmooth convex regression problem}

%\subtitle{Do you have a subtitle?\\ If so, write it here}
\titlerunning{Fast algorithm for $\ell_0$ penalized regression problem}        % if too long for running head

\author{Fan Wu \and
	 Wei Bian$^*$ \and
	 Xiaoping Xue %etc.
}

%\authorrunning{Short form of author list} % if too long for running head

\institute{
School of Mathematics, Harbin Institute of Technology, Harbin, China (Fan Wu) \\
School of Mathematics, Harbin Institute of Technology, Harbin, China; Institute of Advanced Study in Mathematics, Harbin Institute of Technology, Harbin, China (Wei Bian, Xiaoping Xue) \\
             \email{wufanmath@163.com (Fan Wu); bianweilvse520@163.com (Wei Bian); xiaopingxue@hit.edu.cn (Xiaoping Xue)}   \\
           $^*$ Corresponding author
}

\date{Received: date / Accepted: date}
% The correct dates will be entered by the editor

\maketitle

\begin{abstract}\label{abs}
 We investigate a class of constrained sparse regression problem with cardinality penalty, where the feasible set is defined by box constraint, and the loss function is convex, but not necessarily smooth. First, we put forward a smoothing fast iterative hard thresholding (SFIHT) algorithm for solving such optimization problems, which combines smoothing approximations, extrapolation techniques and iterative hard thresholding methods. The extrapolation coefficients can be chosen to satisfy $\sup_k \beta_k=1$ in the proposed algorithm. We discuss the convergence behavior of the algorithm with different extrapolation coefficients, and give sufficient conditions to ensure that any accumulation point of the iterates is a local minimizer of the original cardinality penalized problem. In particular, for a class of fixed extrapolation coefficients, we discuss several different update rules of the smoothing parameter and obtain the convergence rate of $O(\ln k/k)$ on the loss and objective function values. Second, we consider the case in which the loss function is Lipschitz continuously differentiable, and develop a fast iterative hard thresholding (FIHT) algorithm to solve it. We prove that the iterates of FIHT converge to a local minimizer of the problem that satisfies a desirable lower bound property. Moreover, we show that the convergence rate of loss and objective function values are $o(k^{-2})$. Finally, some numerical examples are presented to illustrate the theoretical results.
\end{abstract}

\section{Introduction}
In this paper, we consider the following minimization problem:
\begin{eqnarray}\label{prob1}
\begin{array}{ll}
\mathrm{min}&F(x):= f(x)+\lambda \|x\|_0 \\
\mathrm{s.t.}& x\in \mathcal{X}:=\{x\in \mathbb{R}^n : l\leq x \leq u\},
\end{array}
\end{eqnarray}
for some $\lambda>0$, $l\in {\overline{\mathbb{R}}^{n}_-}:=\{x\in \mathbb{R}^{n}: -\infty\leq x_i \leq 0, 1\leq i \leq n\}$, $u\in {\overline{\mathbb{R}}^{n}_+}:=\{x\in \mathbb{R}^{n}: 0\leq x_i \leq +\infty, 1\leq i \leq n\}$ with $l<u$. In \eqref{prob1}, we call that $f:{{\mathbb{R}}^{n}} \to {\mathbb{R}}$ is the loss function to characterize the data fitting, and $\|x\|_0$ is the penalty function to control the sparsity of solutions. Penalty parameter $\lambda$ is to coordinate the trade-off between the data fitting and sparsity. Throughout this paper, we assume that $f$ in \eqref{prob1} is convex but not necessarily smooth, and we focus on the case that $f$ is nonsmooth.

Such nonsmooth convex regression problems with cardinality penalty arise from many important applications including compressed sensing \cite{Candes2006Robust,Donoho2006Compressed}, variable selection \cite{Liu2007variable}, signal and image processing \cite{Soubies2015a,Bruckstein2009from}, pattern recognition \cite{Blumensath2006sparse} and regression \cite{Tibshirani1996regression}, etc. The purpose of these problems is to find the sparse solutions, most of whose elements are zeros. Owing to the existence of $\ell_0$ function, optimization problem \eqref{prob1} is typical NP-hard in general. When $f$ in \eqref{prob1} is a smooth convex function, a variety of first-order algorithms have been proposed. One important application is the least squares problem, i.e., $f(x)=\|Ax-b\|^2$, with sensing matrix $A\in \mathbb{R}^{m\times n}$ and observation vector $b\in \mathbb{R}^m$. Greedy methods have been proposed to seek the solutions of the $\ell_0$ penalized least squares problems in the early stage, such as matching pursuit (MP) \cite{Mallat1993Matching}, orthogonal matching pursuit (OMP) \cite{Pati2002Orthogonal}, subspace pursuit (SP) \cite{Dai2009subspace}, and so on. With the development of compressed sensing, Donoho \cite{Donoho2006Compressed} and Cand$\mathrm{\grave{e}}$s, Romberg, Tao \cite{Candes2006Robust} confirmed the equivalence of the $\ell_0$ problem and $\ell_1$ problem when $A$ satisfies some proper conditions. Continuous convex relaxation methods is to replace the $\ell_0$ function by a continuous convex function. Even though $\ell_1$ penalty can be used to find the sparse solutions effectively, recent theoretical analysis shows that it often leads to an over-penalized problem or a biased estimator. Therefore, there occur various continuous nonconvex relaxation functions for $\ell_0$ function, such as smoothly clipped absolute deviation (SCAD) function \cite{Fan2001Variable}, hard thresholding function \cite{Zheng2014High}, capped-$\ell_1$ function \cite{Peleg2008A}, transformed $\ell_1$ function \cite{Nikolova2000Local}, etc. It is proved that these continuous nonconvex penalty functions can not only bring accurate sparse solutions, but also reduce the deviation of nonzero elements with respect to the true estimator. Nevertheless, in \cite{Bian2017Optimality}, it has been proved that finding global minimizers of these nonconvex relaxation problems are also NP-hard in general.

Blumensath and Davies \cite{Blumensath2008Iterative} proposed an iterative hard thresholding (IHT) algorithm for solving the unconstrained and constrained $\ell_0$ penalized problems, respectively. And they proved that the iterates converge to a local minimizer when $\|A\|_2<1$. In \cite{Blumensath2009Iterative}, they also verified that the IHT algorithm can obtain an approximated solution if $A$ has restricted isometry property. Moreover, Lu and Zhang \cite{Lu2012Sparse} presented a penalty decomposition method for general $\ell_0$-penalized and $\ell_0$-constrained minimization problem, and established that any accumulation point of the iterates satisfies the first-order optimality conditions. Lu \cite{Lu2014Iterative} also studied an IHT algorithm and its variant for solving (1.1) when $f$ is Lipschitz continuously differentiable.
%In addition, \cite{Yang2017The} proposed a proximal alternating IHT algorithm for solving the $l_0$ problem based on wavelet frame, and established its complexity of $O(1/\sqrt{k})$.

The proximal forward-backward splitting algorithm \cite{Chambolle1998Nonlinear,Combettes2005signal,Daubechies1988An,Hale2008Fixed} is a classical first-order splitting method, which is also called the proximal gradient algorithm. When this method is used to solve the $\ell_1$ penalized convex regression problem, it is often called the iterative shrinkage thresholding algorithm (ISTA). As we know, for the case that the loss function is Lipschitz continuously differentiable and convex, the convergence rate of the objective function values generated by ISTA is $O(k^{-1})$, where $k$ is the iteration counter. Based on ISTA and Nesterov's acceleration scheme, Beck and Teboulle \cite{Beck2009A} proposed a fast iterative shrinkage thresholding algorithm (FISTA), which not only keeps the simplicity and computation cheapness of ISTA, but also improves the convergence rate of the objective function values to $O(k^{-2})$. Moreover, Nesterov \cite{Nesterov2013Gradient} independently proposed an accelerated gradient algorithm for the same problem with the same convergence rate as FISTA. It's noteworthy that Su, Boyd and Cand$\mathrm{\grave{e}}$s \cite{Su2016a} studied the relationships between a second-order ordinary differential equation and the Nesterov's accelerated gradient method. Inspired by the analysis in \cite{Su2016a}, Attouch and Peypouquet \cite{Attouch2016the} investigated a proximal gradient method with extrapolation coefficients $\beta_k=\frac{k-1}{k+\alpha-1}$ for minimizing the sum of a Lipschitz continuously differentiable convex function and a proper closed convex function. In \cite{Attouch2016the}, it is proved that the convergence rate of the objective function values is $o(k^{-2})$, and the iterates $\{x^k\}$ converge to a minimizer of this problem as $\alpha>3$.
%For the same optimization model, Wen and Xue \cite{Wen2019on} established the finite length of the iterates generated by the proximal gradient method with extrapolation if objective function satisfies the uniformized Lojasiewicz property.
Recently, Doikov and Nesterov \cite{Doikov2020contracting} presented a new accelerated algorithm for solving such problem, in which
they used the high-order tensor methods to solve the inner subproblems and gave a complexity estimate under some assumptions. For the case with Lipschitz continuously differentiable but nonconvex loss function, Wen, Chen and Pong \cite{Wen2017Linear} proved that the iterates and objective function values generated by the proximal gradient algorithm with extrapolation are R-linearly convergent under the error bound condition. Later, Adly and Attouch \cite{Adly2020finite} proposed an inertial proximal gradient algorithm with Hessian damping and dry friction to obtain the finite convergence under some certain conditions.

A class of direct methods for solving nonsmooth convex minimization is the subgradient methods. For an $\epsilon >0$, if $x^{\epsilon}$ satisfies $f(x^{\epsilon})-\min f \leq \epsilon$, then $x^{\epsilon}$ is called an $\epsilon$-approximation solution of $\min f$. It has been reported that the complexity of most subgradient methods for finding an $\epsilon$-approximation solution is of the order $O({\epsilon}^{-2})$. For a class of nonsmooth functions with $\max$ operator, Nesterov \cite{Nesterov2005smooth} gave a smooth convex function with Lipschitz continuous gradient of factor ${\epsilon}^{-1}$ to approximate it, where $\epsilon$ is a given and fixed positive paramter. Then, Nesterov \cite{Nesterov2005smooth} proved that the complexity for finding an $\epsilon$-approximation solution of this nonsmooth problem can be improved to $O({\epsilon}^{-1})$
when the accelerated gradient method is applied to solve the approximate smooth function.
Later, the authors in \cite{Hoda2010smoothing} applied the Nesterov's smoothing technique to the Nash equilibria problem and gave a first-order method with the same complexity as \cite{Nesterov2005smooth}. Notably, Chen \cite{Chen2012smoothing} presented a smoothing gradient method for solving the constrained nonsmooth nonconvex minimization problem and demonstrated how to update the smoothing parameter so that the algorithm converges to a stationary point of the problem. Recently, Bian and Chen \cite{Bian2020a} utilized an exact continuous relaxation problem to solve optimization problem \eqref{prob1} and presented a smoothing proximal gradient algorithm, whose iterates are globally convergent to a local minimizer of problem \eqref{prob1} and convergence rate on the objective function values is $o(k^{-\tau})$ with any $\tau\in(0,1/2)$.
In \cite{Attouch2020Newton}, the authors used the proximal regularized inertial Newton algorithm to solve the nonsmooth convex optimization problem, and proved that the convergence rate of the Moreau envelope values of the objective function is $o(k^{-2})$ when the index satisfies an updating rule.

Up to now, very few studies investigated accelerated algorithm for solving problem \eqref{prob1} in any systematic way. Inspired by the good performance of the accelerated algorithm with extrapolation and the smoothing method, we present a smoothing fast iterative hard thresholding (SFIHT) algorithm for solving problem \eqref{prob1}. It is worth emphasizing that the SFIHT algorithm is used to minimize the sum of a nonsmooth convex function and a discontinuous nonconvex function. The strategy of acceleration is to adopt extrapolation on the iterates. As we know, the larger the range of the extrapolation coefficients, the better. It is worth noting that the extrapolation coefficients in our algorithm can satisfy $\sup_k \beta_k =1$. A key technique in the SFIHT algorithm is that we split the range of the extrapolation coefficients into three cases, which are divided in line with the relationships among the $\ell_0$ norms of the newest three adjacent iterates. Besides, though the subproblem in SFIHT algorithm is a nonconvex minimization problem, it has a closed-form solution and can be calculated exactly due to the special structure of $\ell_0$ norm.
So, the proposed SFIHT algorithm is well-defined. We show that the $\ell_0$ norms of the iterates will not change after a finite number of iterations. Then, we discuss the convergence behavior of the SFIHT algorithm with different extrapolation coefficients. Also, we study the case that the loss function in problem \eqref{prob1} is smooth, and show a derived algorithm of SFIHT algorithm (called FIHT algorithm) for solving it. We prove that the iterates of FIHT algorithm converge to a local minimizer of this problem with an important lower bound property. Moreover, we also substantiate that the convergence rate of FIHT algorithm for the corresponding objective and loss function values is $o(k^{-2})$.

\paragraph{\textbf{Contents}} The rest of this paper is organized as follows. In Section \ref{sec:2}, we first review some preliminary results on smoothing method, then we present the SFIHT algorithm for solving nonsmooth nonconvex problem \eqref{prob1}. Next, we analyze the convergence properties of the proposed algorithm with different extrapolation coefficients for solving \eqref{prob1}. In Section \ref{sec:3}, we focus on solving \eqref{prob1} with a smooth convex loss function. When the extrapolation coefficients in SFIHT algorithm are appropriately fixed, we
give a better convergence behaviour of the proposed algorithm for solving this kind of problems.
In Section \ref{sec:4}, we apply the proposed algorithms to some practical instances, and show the value of acceleration by extrapolation in solving problem \eqref{prob1}.
\paragraph{\textbf{Notations}} Throughout this paper, we denote $\mathbb{N}:=\{1,2,\cdots\}$. Let $\mathbb{R}^n$ be the Euclidean space with inner product $\langle\cdot, \cdot\rangle$ and corresponding Euclidean norm $\|\cdot\|$. For vectors $x, y\in \mathbb{R}^n$, $x\geq y$ means that $x_i \geq y_i$, $i=1,2,\cdots,n$. The $\ell_1$ norm of vector $x\in\mathbb{R}^n$ is denoted by $\|x\|_1$, and let $I(x):=\{i:x_i=0\}$. For a matrix $A\in \mathbb{R}^{m\times n}$, we use $A^{\mathrm{T}}$, $\lambda_{\max}(A)$ and $\|A\|=\sqrt{\lambda_{\max}(A^{\mathrm{T}}A)}$ to denote its transpose, largest eigenvalue and spectral norm, respectively. Given a nonempty closed convex set $\Omega\subseteq\mathbb{R}^n$ and a vector $x\in \mathbb{R}^n$, $P_\Omega(x):=\arg\min\{\|x-z\|: z\in\Omega\}$, $N_\Omega(x)$ denotes the normal cone of $\Omega$ at $x$ and $\Omega_J:=\{x\in \Omega:x_j=0, j\in J\}$ for a given index set $J\subseteq\{1,2,\cdots,n\}$.

\section{Numerical algorithm and its convergence analysis}
\label{sec:2}
In this section, we focus on the case that $f$ is a nonsmooth convex function. In what follows, we assume that $f$ is level bounded on $\mathcal{X}$, i.e., set $\{x\in \mathcal{X}: f(x)\leq r\}$ is bounded for any $r\in \mathbb{R}$, which holds naturally if $\mathcal{X}$ is bounded. Note that function $F$ in \eqref{prob1} is level bounded on $\mathcal{X}$ if and only if $f$ is level bounded on $\mathcal{X}$.
\subsection{Smoothing method and basic properties}
To overcome the nondifferentiability of loss function $f$ in \eqref{prob1}, we use a sequence of continuous differentiable functions to approximate $f$.
\begin{definition}\cite{Bian2020a}\label{def:def1}
	We call $\tilde{f} : \mathbb{R}^n \times (0, \bar{\mu}] \to \mathbb{R}$ with $\bar{\mu}>0$ a smoothing function of the convex function $f$ on $\mathcal{X}$, if $\tilde{f}(x, \mu)$ satisfies the following conditions:
	\begin{enumerate}[{\rm (i)}]
		\item for any fixed $\mu > 0$, $\tilde{f}(\cdot, \mu)$ is continuously differentiable in $\mathbb{R}^n$;
		
		\item $\lim_{z\to x, \mu \downarrow 0}\tilde{f}(z,\mu)=f(x), \quad \forall x\in \mathcal{X}$;
		
		\item $\tilde{f}(\cdot,\mu)$ is convex on $\mathcal{X}$ for any fixed $\mu>0$;
		
		\item $\left\{\lim_{z\to x, \mu \downarrow 0} \nabla_z \tilde{f}(z,\mu)\right\}\subseteq \partial f(x), \quad \forall x\in \mathcal{X}$;
		
		\item there exists a positive constant $\kappa$ such that
		\begin{equation*}
		|\tilde{f}(x,\mu_2)-\tilde{f}(x,\mu_1)|\leq \kappa|\mu_1-\mu_2|, \quad \forall x\in \mathcal{X},\ \mu_1, \mu_2 \in (0, \bar{\mu}];
		\end{equation*}
		
		\item there exists a constant $L_{\tilde{f}}>0$ such that for any $\mu \in(0, \bar{\mu}]$, $\nabla_x \tilde{f}(\cdot,\mu)$ is Lipschitz continuous on $\mathcal{X}$ with Lipschitz constant $L_{\tilde{f}}\mu^{-1}$.
	\end{enumerate}
\end{definition}

For the convenience of description, we provide $\tilde{f}$ a smoothing function of $f$ with the definition in Definition \ref{def:def1} in the following analysis and denote $\nabla \tilde{f}(x, \mu)$ the gradient of $\tilde{f}(x, \mu)$ with respect to $x$. By virtue of Definition \ref{def:def1}-(v), we see that
\begin{equation}\label{fka}
|\tilde{f}(x,\mu)-f(x)|\leq \kappa \mu, \quad \forall x\in \mathcal{X},\ 0< \mu \leq \bar{\mu}.
\end{equation}

Smooth approximations for nonsmooth optimization problems have been studied for decades. The fundamental of smoothing method we use in this paper is as follows. We first approximate loss function $f$ by a smooth function with fixed smoothing parameter $\mu$. Then, one find an approximate solution of the following problem
\begin{equation}\label{smf}
\min_{x\in \mathcal{X}}~\tilde{F}(x,\mu):=\tilde{f}(x,\mu)+\lambda \|x\|_0.
\end{equation}
Next, by updating the smoothing parameter $\mu$, we can find a local minimizer of problem \eqref{prob1}.

According to the definition of $\tilde{f}$, for any fixed $\mu \in (0,\bar{\mu}]$, $L \geq L_{\tilde{f}}$, it holds that
\begin{equation}\label{tfl}
\tilde{f}(x, \mu)-\tilde{f}(y, \mu)\leq \langle\nabla \tilde{f}(y, \mu), x-y \rangle +\frac{L}{2\mu}\|x-y\|^2,\quad \forall x, y \in \mathcal{X}.
\end{equation}
Using the convexity of function $f$, we can prove that $x^*\in \mathcal{X}$ is a local minimizer of problem \eqref{prob1} if and only if $x^*$ satisfies
$
0\in [\partial f(x^*)]_i+[N_\mathcal{X}(x^*)]_i, \forall i \notin I(x^*),
$
which is equivalent to
\begin{equation}\label{cocon}
x^* \in \arg\min \{f(x) : x\in \mathcal{X}_{I(x^*)}\}.
\end{equation}
From \eqref{cocon}, we can easily find that any local minimizer of problem \eqref{prob1} has the oracle property \cite{Fan2001Variable}.

\subsection{Smoothing fast iterative hard thresholding (SFIHT) algorithm}

In this subsection, we combine the smoothing method, extrapolation technique and iterative hard thresholding algorithm to present a fast scheme for solving problem \eqref{prob1}. We name it smoothing fast iterative hard thresholding algorithm and denote it by SFIHT algorithm for short.

In order to find an approximate solution of problem \eqref{smf} with a fixed $\mu>0$, we introduce an approximation of $\tilde{F}(x, \mu)$ around the given point $y$ as follows
\begin{equation}
Q(x, y, \mu):=\tilde{f}(y, \mu)+\langle\nabla \tilde{f}(y, \mu), x-y\rangle +\frac{L}{2\mu} \|x-y\|^2+\lambda \|x\|_0
\end{equation}
with a constant $L>L_{\tilde{f}}$. Further, we solve the following optimization problem
\begin{equation}\label{sub}
\min_{x\in \mathcal{X}}~Q(x, y, \mu)
\end{equation}
to find an approximate solution of problem \eqref{smf}.
Although $\ell_0$ function is nonsmooth and nonconvex, the objective function and the constraint set in \eqref{sub} are both separable with respect to all elements of $x$. Using this fact, it has been proved in \cite{Lu2014Iterative} that optimization problem \eqref{sub} has the closed-form solution denoted by $\bar{x}$ and expressed by, for $i=1,2,\ldots,n,$
\begin{equation}\label{solu}
\bar{x}_i=\left\{
\begin{aligned}
%\begin{split}
&\left[P_\mathcal{X}\left(S_L(y,\mu)\right)\right]_i& &\mbox{if}~ \left[S_L(y,\mu)\right]_i^2 - \left[q(y,\mu)\right]_i^2 > \frac{2\lambda\mu}{L},& \\
&0& &\mbox{if}~ \left[S_L(y,\mu)\right]_i^2 - \left[q(y,\mu)\right]_i^2 < \frac{2\lambda\mu}{L},& \\
&\left[P_\mathcal{X}\left(S_L(y,\mu)\right)\right]_i \mbox{or}\ 0&   &\mbox{otherwise,}&
%\end{split}
\end{aligned}\right.
\end{equation}
where $S_L(y,\mu):=y-\frac{\mu}{L}\nabla \tilde{f}(y,\mu)$, $q(y,\mu):=P_\mathcal{X}\left(S_L(y,\mu)\right)-S_L(y,\mu)$. We take problem \eqref{sub} as the unique subproblem of the proposed SFIHT algorithm. See Algorithm \ref{alg:alg1}. Upon the above fact, we know that SFIHT algorithm is well-defined.
\begin{algorithm}
	\caption{Smoothing Fast Iterative Hard Thresholding (SFIHT) algorithm}
	\label{alg:alg1}
	\begin{algorithmic}
		\STATE{\textbf{Initialization:} Take $x^{1}=x^0\in \mathcal{X}$, $L>L_{\tilde{f}}$, $\mu_{1}=\mu_0 \in (0, \bar{\mu}]$ and $\sigma\in\left(0,2\right)$. Set $k=1$.}
		
		\WHILE{a termination criterion is not met,}
		
		\STATE{\begin{enumerate}
				\item [\textbf{Step 1.}] Choose $\beta_k \in \left[0,\sqrt{\frac{\mu_k}{\mu_{k-1}}} \right)$.
				
				\item [\textbf{Step 2.}] Compute
				\begin{align}
				& y^k=x^k+\beta_k(x^k-x^{k-1}),   \label{yk}\\
				& \bar{x}^{k+1}\in {\arg\min}\{Q(x, y^k, \mu_k): x\in \mathcal{X}\}. \label{xbar}
				\end{align}
				
				\item [\textbf{Step 3.}]
				 \textbf{(3a)} If $I(x^{k-1})=I(x^k)=I(\bar{x}^{k+1})$, let $$x^{k+1}=\bar{x}^{k+1}$$ and go to \textbf{Step 4}.
				
			\textbf{(3b)} Otherwise, choose $\beta_k\in \left[0,\sqrt{\frac{L-L_{\tilde{f}}}{4L}\frac{\mu_k}{\mu_{k-1}}}\ \right],$ compute \textbf{Step 2} to obtain $\bar{x}^{k+1}$.
			
				\quad\quad\textbf{(3b-1)} If $I(x^k)=I(\bar{x}^{k+1})$, let $$x^{k+1}=\bar{x}^{k+1}$$\quad\quad and go to \textbf{Step 4}.
				
				\quad\quad\textbf{(3b-2)} Otherwise, choose $\beta_k \in \left[0,\sqrt{\frac{L-L_{\tilde{f}}}{8L-4L_{\tilde{f}}}\frac{\mu_k}{\mu_{k-1}}}\ \right]$, compute \textbf{Step}\\ \quad\quad \textbf{2} to obtain $\bar{x}^{k+1}$ and set $x^{k+1}=\bar{x}^{k+1}$.
				
			\item [\textbf{Step 4.}] Set
			 \begin{equation*}\label{muup}
			\mu_{k+1} = \frac{\mu_0}{(k+2)^{\sigma}}.
			\end{equation*}
			Increment $k$ by one and return to \textbf{Step 1}.
			\end{enumerate}
	        }
		\ENDWHILE
		
		\textbf{Output}\quad  $x^{k}$, $\mu_k$ and $\beta_k$.
	\end{algorithmic}
\end{algorithm}

In each iteration of the SFIHT algorithm, the accelerated iterative hard thresholding method is used to find an approximate solution of problem \eqref{smf}. $\beta_k$ in Step 1 satisfies
$
\beta_k\in \left[0,\ \sqrt{\frac{\mu_k}{\mu_{k-1}}}\right),
$
which is a basic condition for the convergence analysis of the SFIHT algorithm. The new iterate depends on the two previous computed iterates. In order to improve the effect of extrapolation, we divide the extrapolation coefficients into three cases in the algorithm. We adjust the range of extrapolation coefficients according to the relationships among the $\ell_0$ norms of the new computed iterate and two previous iterates. Step 4 is to update the smoothing parameter $\mu_k$, which ensures that $\mu_k$ is decreasing and tends to zero.
\subsection{Convergence analysis}
Let $\{x^k\}$, $\{y^k\}$ and $\{\mu_k\}$ be the output iterates generated by the SFIHT algorithm.
For $\kappa >0$ in Definition \ref{def:def1}-(v), we introduce the following sequence
\begin{equation*}
H(x^{k},\mu_{k},\tau_{k}):=\tilde{F}(x^{k},\mu_{k})+\kappa\mu_{k}+\tau_{k}\|x^{k}-x^{k-1}\|^2,
\end{equation*}
where $\tau_k >0$. Specially, we give a way of choosing $\tau_k$ as follows. For all $k\in\mathbb{N}$,
\begin{equation}\label{tauk}
\tau_k=\left\{
\begin{aligned}
&\frac{L}{4}\mu_{k-1}^{-1}+\frac{L}{4}\beta_k^2\mu_k^{-1}& &\mbox{if}~I(x^{k-1})=I(x^k)=I(x^{k+1}),& \\
&\frac{L-L_{\tilde{f}}}{8}\mu_{k-1}^{-1} & &\mbox{otherwise}.&
\end{aligned}\right.
\end{equation}
In the following, we will use the above $\tau_k$ to analyze the convergence of the SFIHT algorithm, which is the key point for all the following analysis.

First of all, we show that sequence $\{\tau_k\}$ guarantees that sequence $H(x^{k},\mu_{k},\tau_{k})$ is nonincreasing and convergent, i.e., it can be used as an energy function of the SFIHT algorithm. To do this, we first state some basic properties and analyze the relationships between $H(x^k,\mu_k,\tau_k)$ and $H(x^{k+1},\mu_{k+1},\tau_{k+1})$.
\begin{lemma}\label{lem:hk}
		The following statements hold.
	\begin{enumerate}[{\rm (i)}]
		\item For every $k\in \mathbb{N}$, $x^k\in \mathcal{X}$, $\{\mu_k\}$ is monotone decreasing and $\lim_{k\to\infty}\mu_k=0$.% convergent to $0$ as $k$ tends to $\infty$.
	  \item When $I(x^{k})\neq I(x^{k+1})$, we have
		\begin{equation}\label{H1}
		\begin{split}
		& H(x^{k+1},\mu_{k+1},\tau_{k+1})- H(x^{k},\mu_{k},\tau_{k}) \\
		& \leq\left[\tau_{k+1}- \frac{L-L_{\tilde{f}}}{4\mu_k}\right] \|x^{k+1}-x^{k}\|^2 + \left[\frac{2L-L_{\tilde{f}}}{2\mu_k}\beta_k^2-\tau_{k} \right]\|x^{k}-x^{k-1}\|^2.
		\end{split}
		\end{equation}
		
		\item When $I(x^k)=I(x^{k+1})$, we have
		\begin{equation}\label{H2}
		\begin{split}
		&H(x^{k+1},\mu_{k+1},\tau_{k+1})- H(x^{k},\mu_{k},\tau_{k}) \\
		&\leq  \left[\tau_{k+1}-\frac{L}{2\mu_k}\right] \|x^{k+1}-x^k\|^2
		+ \left[\frac{L}{2\mu_k}\beta_k^2-\tau_{k}\right]\|x^{k}-x^{k-1}\|^2.
		\end{split}
		\end{equation}
	\end{enumerate}
\end{lemma}
\begin{proof}
$\mbox{(i).~}$ By the proposed SFIHT algorithm, it's easy to verify this statement.

$\mbox{(ii).~}$ Using \eqref{tfl} with $x=x^{k+1}$, $y=y^k$ and $\mu=\mu_k$, we obtain
\begin{equation}\label{gra}
\begin{split}
&\tilde{f}(x^{k+1},\mu_k)\leq \tilde{f}(y^k,\mu_k) + \langle \nabla \tilde{f}(y^k,\mu_k),x^{k+1}-y^k \rangle + \frac{L_{\tilde{f}}}{2\mu_k} \|x^{k+1}-y^k\|^2 .\\
\end{split}
\end{equation}
Since $\tilde{f}(x,\mu)$ is convex with respect to $x$ for any fixed $\mu\in(0, \bar{\mu}]$, it holds that
\begin{equation}\label{con}
\tilde{f}(y^k,\mu_k)+ \langle \nabla \tilde{f}(y^k,\mu_k),x-y^k \rangle\leq\tilde{f}(x,\mu_k), \quad \forall x\in \mathcal{X}.
\end{equation}
According to \eqref{gra}, for any $x\in \mathcal{X}$, $L>L_{\tilde{f}}$, we have
\begin{equation*}
\begin{split}
& \tilde{F}(x^{k+1},\mu_k)
= \tilde{f}(x^{k+1},\mu_k)+\lambda \|x^{k+1}\|_0 \\
\leq & \tilde{f}(y^k,\mu_k) + \langle \nabla \tilde{f}(y^k,\mu_k),x^{k+1}-y^k \rangle +  \frac{L}{2\mu_k} \|x^{k+1}-y^k\|^2 +\lambda \|x^{k+1}\|_0  \\
 &+ \frac{L_{\tilde{f}} - L}{2\mu_k}\|x^{k+1}-y^k\|^2.
\end{split}
\end{equation*}
This, combined with the definition of $x^{k+1}$, yields that for $x^k\in \mathcal{X}$,
\begin{equation*}
\begin{split}
\tilde{F}(x^{k+1},\mu_k)\leq & \tilde{f}(y^k,\mu_k) + \langle \nabla \tilde{f}(y^k,\mu_k),x^k-y^k \rangle +  \frac{L}{2\mu_k} \|x^k-y^k\|^2 +\lambda \|x^k\|_0 \\
& + \frac{L_{\tilde{f}} - L}{2\mu_k} \|x^{k+1}-y^k\|^2  \\
\leq & \tilde{F}(x^k,\mu_k)+ \frac{L}{2\mu_k} \|x^k-y^k\|^2 + \frac{L_{\tilde{f}} - L}{2\mu_k} \|x^{k+1}-y^k\|^2, \\
\end{split}
\end{equation*}
where the last inequality holds by \eqref{con}.
The above inequality together with $y^{k}=x^{k}+\beta_k(x^k-x^{k-1})$, gives
\begin{equation}\label{ftd}
\begin{split}
\tilde{F}(x^{k+1},\mu_k)\leq & \tilde{F}(x^{k},\mu_k)+ \frac{L_{\tilde{f}}}{2\mu_k} \beta_k^2\|x^k-x^{k-1}\|^2 + \frac{L_{\tilde{f}} - L }{2\mu_k}\|x^{k+1}-x^{k}\|^2 \\
& + \frac{L - L_{\tilde{f}}}{\mu_k}\langle x^{k+1}-x^{k}, \beta_k(x^k-x^{k-1})\rangle. \\
\end{split}
\end{equation}
Using the algebratic inequality, it holds that
\begin{equation*}
\frac{L - L_{\tilde{f}}}{\mu_k}\langle x^{k+1}-x^{k}, \beta_k(x^k-x^{k-1})\rangle \leq \frac{L - L_{\tilde{f}}}{4\mu_k}\|x^{k+1}-x^{k}\|^2
+ \frac{L - L_{\tilde{f}}}{\mu_k}\beta_k^2\|x^{k}-x^{k-1}\|^2.
\end{equation*}
Combining this with \eqref{ftd}, we obtain
\begin{equation}\label{ft}
\begin{split}
\tilde{F}(x^{k+1},\mu_k) \leq \tilde{F}(x^{k},\mu_k) &- \frac{L -L_{\tilde{f}}}{4\mu_k} \|x^{k+1}-x^{k}\|^2 +\frac{2L-L_{\tilde{f}}}{2\mu_k}\beta_k^2\|x^{k}-x^{k-1}\|^2.
\end{split}
\end{equation}
By Definition \ref{def:def1}-(v) and the monotone decreasing of $\{\mu_k\}$, we have
\begin{equation}\label{fne}
\tilde{F}(x^{k+1},\mu_{k+1})+\kappa \mu_{k+1}-\kappa\mu_k\leq \tilde{F}(x^{k+1},\mu_{k}),
\end{equation}
plugging \eqref{fne} in \eqref{ft}, then we get
\begin{equation}\label{f}
\begin{split}
&\tilde{F}(x^{k+1},\mu_{k+1})+\kappa \mu_{k+1} \\
\leq &\tilde{F}(x^{k},\mu_{k}) + \kappa \mu_{k}- \frac{L -L_{\tilde{f}}}{4\mu_k} \|x^{k+1}-x^{k}\|^2+ \frac{2L - L_{\tilde{f}}}{2\mu_k}\beta_k^2\|x^{k}-x^{k-1}\|^2.
\end{split}
\end{equation}
Using \eqref{f} and the definition of $H(x^{k},\mu_{k},\tau_{k})$, we obtain \eqref{H1} immediately.

$\mbox{(iii).~}$ Let $$G(x,y,\mu):=\tilde{f}(y,\mu) + \langle \nabla \tilde{f}(y,\mu),x-y \rangle + \frac{L}{2\mu} \|x-y\|^2.$$
For each fixed $y$ and $\mu$, $G(x,y,\mu)$ is differentiable and strongly convex with respect to $x$ with modulus $L\mu^{-1}$.
Hence, for arbitrary $x \in \mathcal{X}$, we have
\begin{equation*}
	\begin{split}
		& \langle\nabla_x G(x^{k+1},y^k,\mu _k),x-x^{k+1}\rangle \\
		=& \langle \nabla \tilde{f}(y^k,\mu_k),x-x^{k+1} \rangle + \frac{L}{2\mu_k} \|x-y^k\|^2-\frac{L }{2\mu_k} \|x^{k+1}-x\|^2-\frac{L}{2\mu_k} \|x^{k+1}-y^k\|^2. \\
	\end{split}
\end{equation*}
When $I(x^k)=I(x^{k+1})$, in view of the construction of the SFIHT algorithm, we find
\begin{equation}\label{sep}
x^{k+1}= \arg\min_{x\in \mathcal{X}}\{G(x,y^k,\mu_k)+\lambda\|x^k\|_0\}.
\end{equation}
By the above fact, we obtain that
$\langle\nabla_x G(x^{k+1},y^k,\mu _k),x-x^{k+1}\rangle \geq 0$, $\forall x \in \mathcal{X},$
which indicates that, for all $ x\in \mathcal{X}$,
\begin{equation}\label{nc}
0\leq \langle \nabla \tilde{f}(y^k,\mu_k),x-x^{k+1} \rangle -\frac{L}{2\mu_k}\|x^{k+1}-y^k\|^2+ \frac{L}{2\mu_k} \|x-y^k\|^2-\frac{L}{2\mu_k} \|x^{k+1}-x\|^2.
\end{equation}
Summing up \eqref{gra} and \eqref{nc}, and by $L>L_{\tilde{f}}$, we have
\begin{equation}\label{rL}
\begin{split}
&\tilde{f}(x^{k+1},\mu_k) \\
\leq & \tilde{f}(y^k,\mu_k)+\langle \nabla \tilde{f}(y^k,\mu_k),x-y^k \rangle + \frac{L }{2\mu_k} \|x-y^k\|^2 -\frac{L}{2\mu_k} \|x^{k+1}-x\|^2.
\end{split}
\end{equation}
Substituting \eqref{con} into the right side of \eqref{rL}, one has
\begin{equation}\label{ftid}
\tilde{f}(x^{k+1},\mu_k) \leq \tilde{f}(x,\mu_k) + \frac{L}{2\mu_k} \|x-y^k\|^2-\frac{L}{2\mu_k} \|x^{k+1}-x\|^2.
\end{equation}
Letting $x=x^k$ in \eqref{ftid}, and using the definition of $y^{k}$, $\|x^k\|_0=\|x^{k+1}\|_0$ and \eqref{fne}, we obtain
\begin{equation}\label{fe}
\begin{split}
&\tilde{F}(x^{k+1},\mu_{k+1})+\kappa\mu_{k+1} \\
\leq &\tilde{F}(x^k,\mu_{k})+\kappa\mu_{k}+ \frac{L}{2\mu_k}\beta_k^2 \|x^k-x^{k-1}\|^2-\frac{L}{2\mu_k} \|x^{k+1}-x^k\|^2.
\end{split}
\end{equation}
By the definition of $H(x^{k},\mu_{k},\tau_{k})$ and \eqref{fe}, we see the statement in item (iii).
\end{proof}

For simplicity, we define the notations
\begin{equation}\label{dk}
\gamma:=\min \left\{\frac{L}{4}, \frac{L-L_{\tilde{f}}}{8}\right\},\quad \mathcal{K}:=\{k: I(x^k)\neq I(x^{k+1})\}
\end{equation}
and
\begin{equation}\label{nu}
	\nu:= \min\left\{l_i^2\mu_0^{-1},u_j^2\mu_0^{-1},2\lambda L^{-1}: l_i\neq 0, u_j\neq 0,\ i,j\in\{1,2,\ldots,n\}\right\}.%=1,2,\ldots,n, j=1,2,\ldots,n\}.
\end{equation}

\begin{lemma}\label{lem:hnin}
		The following statements hold.
	\begin{enumerate}[{\rm (i)}]
		\item When $k\in\mathcal{K}$, $\frac{2L-L_{\tilde{f}}}{2\mu_{k}}\beta_{k}^2\leq \tau_{k}$; otherwise, $\frac{L}{2\mu_{k}}\beta_{k}^2 \leq \tau_{k}$.
   \item $\{H(x^{k},\mu_{k},\tau_{k})\}$ is nonincreasing and convergent, i.e., $\lim_{k\to\infty}H(x^k,\mu_k,\tau_k) =H_{\infty} <\infty$.
   \end{enumerate}
\end{lemma}

\begin{proof}
	$\mbox{(i).~}$
	For $k\in \mathcal{K}$, there must be $\beta_k^2\leq \frac{L-L_{\tilde{f}}}{8L-4L_{\tilde{f}}}\frac{\mu_k}{\mu_{k-1}}$, which means that
	$\tau_k=\frac{L-L_{\tilde{f}}}{8}\mu_{k-1}^{-1}\geq\frac{2L-L_{\tilde{f}}}{2}\beta_k^2\mu_k^{-1}$.
	
	 $k\notin\mathcal{K}$ implies $k\in \mathcal{N}_1:=\{k : I(x^{k-1})=I(x^k)=I(x^{k+1})\}$ or $k\in \mathcal{N}_2 := \{k: I(x^{k-1})\neq I(x^k)= I(x^{k+1})\}$. From the SFIHT algorithm, we know $\beta_k^2<\frac{\mu_k}{\mu_{k-1}}$ for $k\in\mathcal{N}_1$ and $\beta_k^2\leq\frac{L-L_{\tilde{f}}}{4L}\frac{\mu_k}{\mu_{k-1}}$ for $k\in\mathcal{N}_2$. Then, for $k\in\mathcal{N}_1$, $\mu_{k-1}^{-1}>\beta_k^2\mu_k^{-1}$, and by \eqref{tauk}, we have
	\begin{equation*}
		\tau_k=\frac{L}{4}\mu_{k-1}^{-1}+\frac{L}{4}\beta_k^2\mu_k^{-1}>\frac{L}{4}\beta_k^2\mu_k^{-1}+\frac{L}{4}\beta_k^2\mu_k^{-1}=\frac{L}{2}\beta_k^2\mu_k^{-1};
		\end{equation*}
	for $k\in\mathcal{N}_2$, $\mu_{k-1}^{-1}\geq \frac{4L}{L-L_{\tilde{f}}}\beta_k^2\mu_k^{-1}$, and by \eqref{tauk}, we have
	$\tau_k=\frac{L-L_{\tilde{f}}}{8}\mu_{k-1}^{-1}\geq  \frac{L}{2}\beta_k^2\mu_k^{-1}.$
	 Then, we establish result (i).
	
$\mbox{(ii).~}$ We first analyze the nonincreasing of $\{H(x^{k},\mu_{k},\tau_{k})\}$, and divide the proof into two cases as follows.

Case 1. If $k\in \mathcal{K}$, by \eqref{tauk}, we have $\tau_{k+1}=\frac{L-L_{\tilde{f}}}{8}\mu_{k}^{-1}$. This together with \eqref{H1} and (i) of this lemma, it holds that
\begin{equation}\label{k1}
 H(x^{k+1},\mu_{k+1},\tau_{k+1})- H(x^{k},\mu_{k},\tau_{k})
 \leq - \frac{L-L_{\tilde{f}}}{8}\mu_k^{-1} \|x^{k+1}-x^{k}\|^2,\quad\forall k\in\mathcal{K}.
\end{equation}

Case 2. If $k\notin \mathcal{K}$, we know that $k\in\{I(x^{k})=I(x^{k+1})=I(x^{k+2})\}$ or $k\in\{I(x^{k})=I(x^{k+1})\neq I(x^{k+2})\}$.

When $k\in\{I(x^{k})=I(x^{k+1})=I(x^{k+2})\}$, $\tau_{k+1}=\frac{L}{4}\mu_{k}^{-1}+\frac{L}{4}\beta_{k+1}^2\mu_{k+1}^{-1}$. By \eqref{H2} and statement (i) in this lemma, we
obtain
 \begin{equation}\label{cas21}
 	\begin{split}
 		&H(x^{k+1},\mu_{k+1},\tau_{k+1})- H(x^{k},\mu_{k},\tau_{k}) \\
 		&\leq \left[\tau_{k+1}-\frac{L}{2\mu_{k}}\right] \|x^{k+1}-x^{k}\|^2
 		=-\frac{L}{4}\left[1-\beta_{k+1}^2\frac{\mu_{k}}{\mu_{k+1}}\right]\mu_{k}^{-1}\|x^{k+1}-x^{k}\|^2.
 	\end{split}
 \end{equation}
By $\beta_{k+1} \in\left[0,\sqrt{\frac{\mu_{k+1}}{\mu_{k}}}\right)$ in this case, we obtain $H(x^{k+1},\mu_{k+1},\tau_{k+1})\leq H(x^{k},\mu_{k},\tau_{k}).$

When $k\in\{I(x^{k})=I(x^{k+1})\neq I(x^{k+2})\}$, by \eqref{tauk}, we have $\tau_{k+1}=\frac{L-L_{\tilde{f}}}{8}\mu_{k}^{-1}$. According to \eqref{H2} and result (i) of this lemma, it yields that
\begin{equation}\label{cas22}
\begin{split}
H(x^{k+1},\mu_{k+1},\tau_{k+1})- H(x^{k},\mu_{k},\tau_{k}) \leq  - \frac{3L +L_{\tilde{f}}}{8} \mu_k^{-1} \|x^{k+1}-x^k\|^2.
\end{split}
\end{equation}

Hence, $\{H(x^{k},\mu_{k},\tau_{k})\}$ is nonincreasing. Since $f$ is bounded from below on $\mathcal{X}$ and $\{x^k\}\subseteq \mathcal{X}$, $\{H(x^{k},\mu_{k},\tau_{k})\}$ is also bounded from below on $\mathcal{X}$. This together with the nonincreasing of $\{H(x^{k},\mu_{k},\tau_{k})\}$ implies that $\{H(x^{k},\mu_{k},\tau_{k})\}$ is convergent.
\end{proof}
	
The next lemma explores the boundedness of sequence $\{x^k\}$, and gives an estimate on $\{x^k\}$ and $\{\mu_k\}$, which lays a foundation for the analysis of $I(x^k)$.
\begin{lemma}\label{lem:hcon}
The following statements hold:
	\begin{enumerate}[{\rm (i)}]
		\item sequence $\{x^k\}$ is bounded;
		
		\item for any $k\in\mathcal{K}$, $\|x^{k+1}-x^k\|^2\geq \nu \mu_k$, where $\nu$ is defined as in \eqref{nu}.
	\end{enumerate}
\end{lemma}

\begin{proof}
$\mbox{(i).~}$ By \eqref{fka} and the definition of $\{H(x^{k},\mu_{k},\tau_{k})\}$, we have
\begin{equation*}
\begin{split}
F(x^k)& \leq \tilde{F}(x^k, \mu_{k})+\kappa \mu_{k} \leq H(x^k, \mu_{k}, \tau_k) \leq H(x^1, \mu_{1}, \tau_1) =\tilde{F}(x^1, \mu_{1})+\kappa \mu_{1},
\end{split}
\end{equation*}
which together with the level boundedness of $F$ on $\mathcal{X}$ gives result (i).
%the boundedness of $\{x^k\}$.

$\mbox{(ii).~}$
For any $k\in \mathcal{K}$, from \eqref{solu}, there exists some $i\in \{1,2,\ldots,n\}$ such that $$x^{k}_i=[P_\mathcal{X}(S_L(y^{k-1},\mu_{k-1}))]_i\neq 0,\ x^{k+1}_i=0$$ or $$ x^{k}_i=0,\ x^{k+1}_i=[P_\mathcal{X}(S_L(y^{k},\mu_{k}))]_i\neq 0.$$
When $x^{k}_i=[P_\mathcal{X}(S_L(y^{k-1},\mu_{k-1}))]_i\neq 0$ and $x^{k+1}_i=0$,
by the definition of $P_\mathcal{X}(\cdot)$, there exist three cases: (a) $x^{k}_i = l_i < 0$; (b) $x^{k}_i = u_i > 0$; (c) $x^{k}_i=[S_L(y^{k-1},\mu_{k-1})]_i$. For case (a), we have
\begin{equation}\label{con1}
\|x^{k+1}-x^{k}\|^2 \geq |x^{k+1}_i-x^{k}_i|^2=|x^{k}_i|^2=|l_i|^2 \geq \nu\mu_0.
\end{equation}
For case (b),  similar to the analysis in case (a), we see that \eqref{con1} also holds.
For case (c), according to \eqref{solu}, we have $\left[S_L(y^{k-1},\mu_{k-1})\right]^2_i \geq 2\lambda L^{-1}\mu_{k-1}$. Hence,
\begin{equation*}
	\|x^{k+1}-x^{k}\|^2 \geq |x^{k+1}_i-x^{k}_i|^2=|x^{k}_i|^2=\left[S_L(y^{k-1},\mu_{k-1})\right]^2_i \geq \nu\mu_{k-1}.
\end{equation*}
These together with the decreasing of $\{\mu_k\}$, yields that
\begin{equation}\label{con2}
\|x^{k+1}-x^{k}\|^2 \geq \nu\mu_{k}.
\end{equation}
When $x^{k}_i=0$ and $x^{k+1}_i=[P_\mathcal{X}(S_L(y^{k},\mu_{k}))]_i\neq 0$, by the similar arguments, we can get the same result as above. Then, we have thus proved this statement.
\end{proof}

Now we show that the support set of $\{x^k\}$ generated by the SFIHT algorithm will no longer change after finitely many iterations, and give some other useful estimates.
\begin{lemma}\label{lem:ixk}
 The sequences generated by the SFIHT algorithm own the following properties:
	\begin{enumerate}[{\rm (i)}]
		
		\item $I(x^k)=\{i:x^k_i=0\}$ changes finite times at most;
		
		\item the output extrapolation coefficients can be chosen to satisfy $\sup_k\beta_k=1$;
		
		\item $\sum_{k=1}^{\infty}\left[1-\beta_{k+1}^2\frac{\mu_{k}}{\mu_{k+1}}\right]
		\mu_{k}^{-1}\|x^{k+1}-x^{k}\|^2<\infty;$
		
		\item $\lim_{k\to\infty}\left[f(x^k)+\tau_k\|x^k-x^{k-1}\|^2\right]$ exists.
	\end{enumerate}
\end{lemma}
\begin{proof}
$\mbox{(i).~}$ Recalling $\mathcal{K}=\{k: I(x^k)\neq I(x^{k+1})\},$ then, we only need to show that set $\mathcal{K}$ has at most finite elements. We argue it by contradiction and suppose there are infinite elements in $\mathcal{K}$. This, together with \eqref{k1} and (ii) of Lemma \ref{lem:hnin}, we have
\begin{equation}\label{k1c1}
\begin{split}
0 &\leq \sum_{k\in \mathcal{K}} \frac{1}{8}(L - L_{\tilde{f}})\mu_{k}^{-1} \|x^{k+1}-x^{k}\|^2 \\
&\leq \sum_{k=1}^{\infty} \left[H(x^{k},\mu_{k},\tau_{k}) - H(x^{k+1},\mu_{k+1},\tau_{k+1})\right] =H(x^{1},\mu_{1},\tau_{1}) -H_{\infty} < \infty.
\end{split}
\end{equation}
On the basis of Lemma \ref{lem:hcon}-(ii), we obtain
\begin{equation*}
\begin{split}
\sum_{k\in \mathcal{K}} \frac{1}{8}(L - L_{\tilde{f}})\mu_{k}^{-1} \|x^{k+1}-x^{k}\|^2 &\geq\sum_{k\in \mathcal{K}} \frac{1}{8}(L - L_{\tilde{f}})\nu =\infty.
\end{split}
\end{equation*}
This leads to a contradiction to \eqref{k1c1}.
Hence, set $\mathcal{K}$ has at most finite elements, we see further that $I(x^k)=\{i:x^k_i=0\}$ changes finite times at most. % Thus, we complete the proof of statement (i).

$\mbox{(ii)}.~$
In view of result (i) of this lemma, we know that the SFIHT algorithm will continue to run (3a) in Step 3 after finite iterations. Then, by the update rule of $\mu_k$ in Step 4, the statement in (ii) holds.

$\mbox{(iii).~}$
According to \eqref{k1}, \eqref{cas21} and \eqref{cas22}, and by \eqref{dk}, for $\forall k\in\mathbb{N}$, we find
\begin{equation*}\label{hk2}
H(x^{k+1},\mu_{k+1},\tau_{k+1})- H(x^{k},\mu_{k},\tau_{k})\leq -\gamma\left(1-\beta_{k+1}^2\frac{\mu_{k}}{\mu_{k+1}}\right)\mu_{k}^{-1} \|x^{k+1}-x^{k}\|^2.
\end{equation*}
Summing up the above inequality from $1$ to $\infty$ and by Lemma \ref{lem:hnin}-(ii), we obtain
\begin{equation*}%\label{hkn}
\begin{split}
0\leq & \sum_{k=1}^{\infty}\gamma\left(1-\beta_{k+1}^2\frac{\mu_{k}}{\mu_{k+1}}\right)\mu_{k}^{-1} \|x^{k+1}-x^{k}\|^2 \\
\leq &\sum_{k=1}^{\infty} \left[H(x^{k},\mu_{k},\tau_{k})- H(x^{k+1},\mu_{k+1},\tau_{k+1})\right] = H(x^{1},\mu_1,\tau_{1})-H_{\infty}<\infty.
\end{split}
\end{equation*}
In view of the definition of $\gamma$ in \eqref{dk}, we get the desired result (iii).

$\mbox{(iv)}.~$
Taking $x=x^{k}$ and $\mu =\mu_{k}$ in \eqref{fka}, and by direct computation, it yields that
\begin{equation*}\label{fx0}
\begin{split}
H(x^{k},\mu_{k},\tau_{k})-2\kappa \mu_{k}
\leq f(x^{k})+\lambda\|x^{k}\|_0+\tau_k\|x^k-x^{k-1}\|^2
\leq H(x^{k},\mu_{k},\tau_{k}).
\end{split}
\end{equation*}
Letting $k$ tend to infinity in the above inequality, along with $\lim_{k\to \infty}\mu_{k}=0$, we get
\begin{equation*}
	\lim_{k\to\infty}f(x^{k})+\lambda\|x^{k}\|_0+\tau_k\|x^k-x^{k-1}\|^2
	=\lim_{k\to \infty}H(x^{k},\mu_{k},\tau_{k}).
\end{equation*}
This, combined with the fact that $\lim_{k\to\infty}\|x^k\|_0$ exists, we deduce the existence of
\begin{equation*}
\lim_{k\to\infty}\left[f(x^{k})+\tau_k\|x^k-x^{k-1}\|^2\right].
\end{equation*}
\end{proof}

Refs. \cite{Lu2014Iterative} and \cite{Wu2020accelerated} also study the IHT algorithm for solving the constrained $\ell_0$ penalized convex regression problem modeled by \eqref{prob1}. In terms of problem, the main difference is that the loss functions studied by them are smooth, while it can be nonsmooth in this paper. In terms of algorithm, \cite{Lu2014Iterative} considers the IHT algorithm, while both \cite{Wu2020accelerated} and this paper focus on the IHT algorithm with extrapolation. It's worth stressing that the extrapolation coefficients in \cite{Wu2020accelerated} need satisfy $\sup_k \beta_k\leq\frac{\sqrt{2}}{2}$, but the SFIHT algorithm proposed in this paper expands the range of extrapolation coefficients in a significant way, which can be seen clearly by the following results on convergence rate.

\begin{remark}
	If $\beta_k \in\left[0, \sqrt{(1-a\mu_{k-1})\frac{\mu_{k}}{\mu_{k-1}}}\right]$ with $a>0$ in Step 1, by Lemma  \ref{lem:ixk}-(iii), we obtain $\sum_{k=1}^\infty \|x^{k+1}-x^k\|^2<\infty$.
	\end{remark}

Combining Lemma \ref{lem:ixk}-(i) and the framework of the SFIHT algorithm, we know that the algorithm only runs (3a) in Step 3 after a finite number of iterations, which means that the SFIHT algorithm solves subproblem \eqref{xbar} only one time and is used to solve a nonsmooth convex optimization after finite itarations. Hence, in order to improve the convergence behavior of the iterates generated by the SFIHT algorithm, we will consider two different choices of $\beta_k$ in Step 1 of the SFIHT algorithm. %which are discussed in two cases below.
%which are divided into two cases in the following.

\subsubsection{A sufficient condition for the convergence of $\{x^k\}$}
In this subsubsection, we will analyze the convergence of the iterates generated by the SFIHT algorithm for solving \eqref{prob1} when
\begin{equation}\label{sig1}
\sigma\in\left[\frac{1}{2}, 1\right]
\end{equation}
and
\begin{equation}\label{be1}
 \beta_k\in\left[0, \sqrt{\left(1-\frac{1}{2k^{1-\sigma}}\right)\frac{\mu_k}{\mu_{k-1}}}\ \right]
\end{equation}
in Step 1.
It's easy to find that $\beta_k<\sqrt{\frac{\mu_{k}}{\mu_{k-1}}}$, hence, the previous results are still hold in this case. Based on these results, we first give the following estimations.
%which are the key to the analysis of the forthcoming \cref{the:theorem1}.

\begin{lemma}\label{lem:keyl}
	When $\sigma$ and $\beta_k$ in Step 1 are chosen as in \eqref{sig1} and \eqref{be1},
	it holds that
	$$\sum_{k=1}^{\infty}\frac{1}{k+1}\mu_k^{-2}\|x^{k+1}-x^{k}\|^2<\infty\quad\mbox{and}\quad \lim_{k\to\infty}\|x^{k+1} - x^k\|=0.$$
	\end{lemma}
\begin{proof}
 From Lemma \ref{lem:ixk}-(i) and \eqref{be1}, there exists $K\in\mathbb{N}$, such that
 $$\beta_{k+1}^2\leq \left(1-\frac{1}{2(k+1)^{1-\sigma}}\right)\frac{\mu_{k+1}}{\mu_{k}},\quad \forall k\geq K.$$
This together with Lemma \ref{lem:ixk}-(iii), we obtain
%\begin{equation*}%\label{esta1}
$	\sum_{k=1}^{\infty}\frac{1}{2(k+1)^{1-\sigma}}
	\mu_{k}^{-1}\|x^{k+1}-x^{k}\|^2<\infty.$
%\end{equation*}
According to $\mu_k=\frac{\mu_0}{(k+1)^\sigma}$, we see that
\begin{equation*}
\sum_{k=1}^\infty \frac{1}{k+1}\mu_k^{-2}\|x^{k+1}-x^k\|^2<\infty\quad\mbox{and}\quad \sum_{k=1}^{\infty}(k+1)^{2\sigma-1}\|x^{k+1}-x^{k}\|^2<\infty.
\end{equation*}
Thanks to $\sigma \in [\frac{1}{2}, 1]$, we get
$
	\sum_{k=1}^{\infty}\|x^{k+1}-x^k\|^2<\infty.
	%\leq \sum_{k=1}^{\infty}(k+1)^{2\sigma-1}\|x^{k+1}-x^k\|^2<\infty.
$
Then, we have got all the estimates in this lemma.	
\end{proof}

Next, we give another preliminary result.

\begin{proposition}\label{pro:prop1}
	Let $\{a_k\}$ and $\{b_k\}$ be nonnegative sequences with $\sum_{k=1}^{\infty}a_kb_k<\infty$. If $\{a_k\}$ is a nonincreasing sequence and satisfies $\sum_{k=1}^{\infty}a_k=\infty$, then there exists a subsequence of $\{b_k\}$, denoted by $\{b_{k_i}\}$, satisfying $\lim_{i\to\infty}b_{k_i-1}=0$ and $\lim_{i\to\infty}b_{k_i}=0$.
\end{proposition}
\begin{proof}
Since sequence $\{a_k\}$ is nonincreasing, we have
\begin{equation*}
\sum_{k=2}^{\infty}a_{k}(b_{k-1}+b_{k})
\leq\sum_{k=2}^{\infty}a_{k-1}b_{k-1}+\sum_{k=2}^{\infty}a_{k}b_{k}<\infty.
\end{equation*}
This together with $\sum_{k=2}^{\infty}a_{k}=\infty$ implies that there exists a subsequence of $\{b_{k-1}+b_{k}\}$, denoted by$\{b_{k_i-1}+b_{k_i}\}$, such that  $\lim_{i\to\infty}(b_{k_i-1}+b_{k_i})=0$. Due to the nonnegativity of $\{b_k\}$, it follows that $\lim_{i\to\infty}b_{k_i-1}=0$ and $\lim_{i\to\infty}b_{k_i}=0$.
\end{proof}

\begin{theorem}\label{the:theorem1}
	Any accumulation point of sequence $\{x^k\}$ generated by the SFIHT algorithm is a local minimizer of problem \eqref{prob1}.
\end{theorem}
\begin{proof}
From Lemma \ref{lem:ixk}-(i), we know that there exist some $K\in \mathbb{N}$ and $I\subseteq\{1,2,\ldots,n\}$ such that $I(x^k)=I$ for all $k\geq K$. This, combined with the SFIHT algorithm, we have
\begin{equation}\label{xn0}
x^{k+1}=\arg\min_{x\in \mathcal{X}_I}\left\{\tilde{f}(y^k,\mu_k) + \langle \nabla \tilde{f}(y^k,\mu_k),x-y^k \rangle + \frac{L}{2\mu_k} \|x-y^k\|^2\right\}, \quad \forall k\geq K.
\end{equation}
By Lemma \ref{lem:keyl}, we know
\begin{equation*}
\sum_{k=1}^\infty\frac{1}{k+1}\left[\mu_k^{-2}\|x^{k+1}-x^{k}\|^2\right]<\infty.
\end{equation*}
Using the above result, Proposition \ref{pro:prop1} with $a_k=\frac{1}{k+1}$ and $b_k=\mu_k^{-2}\|x^{k+1}-x^{k}\|^2$, there exists a subsequence of $\{\mu_k^{-2}\|x^{k+1}-x^{k}\|^2\}$ satisfying
\begin{equation}\label{muxk}
\lim_{i\to\infty}\mu_{k_i-1}^{-2}\|x^{k_i}-x^{k_i-1}\|^2=0\quad \mbox{and}\quad \lim_{i\to\infty}\mu_{k_i}^{-2}\|x^{k_i+1}-x^{k_i}\|^2=0.
\end{equation}
We see further that $\lim_{i\to\infty}\mu_{k_i-1}^{-1}(x^{k_i}-x^{k_i-1})=0$ and $\lim_{i\to\infty}\mu_{k_i}^{-1}(x^{k_i+1}-x^{k_i})=0,$
which implies
\begin{equation}\label{muk0}
\lim_{i\to\infty}\frac{x^{k_i+1}-y^{k_i}}{\mu_{k_i}}=\lim_{i\to\infty}\left[\frac{x^{k_i+1}-x^{k_i}}{\mu_{k_i}}
-\beta_{k_i}\frac{\mu_{k_i-1}}{\mu_{k_i}}\frac{x^{k_i}-x^{k_i-1}}{\mu_{k_i-1}}\right]=0.
\end{equation}
From Lemma \ref{lem:hcon}-(i) and Lemma \ref{lem:ixk}-(i), we know that there exists a subsequence of $\{x^{k_i}\}$ (also denoted by $\{x^{k_i}\}$ for simplicity) and $\bar{x}\in \mathcal{X}_I$ such that $\lim_{i\to \infty} x^{k_i}=\bar{x}$.
By the triangle inequality, it holds that
\begin{equation*}
\|x^{k_i+1}-\bar{x}\| \leq \|x^{k_i+1}-x^{k_i}\|+\|x^{k_i}-\bar{x}\|,
\end{equation*}
then we immediately deduce that $\lim_{i\to \infty} x^{k_i+1}=\bar{x}$ by $\lim_{i\to\infty}\|x^{k+1}-x^{k}\|= 0$ proved in Lemma \ref{lem:keyl}. According to the definitions of $y^{k_i}$ and $\beta_{k_i}$, we can also obtain $\lim_{i\to \infty}y^{k_i}=\bar{x}$, and by Definition
\ref{def:def1}-(iv), we know that
\begin{equation}\label{par}
\left\{\lim_{i\to\infty} \nabla \tilde{f}(y^{k_i},\mu_{k_i})\right\}\subseteq \partial f(\bar{x}).
\end{equation}

For any $k_i\geq K$, recalling the definition of $x^{k_i+1}$ in \eqref{xn0}, we have
\begin{equation}\label{fon}
\left\langle \nabla \tilde{f}(y^{k_i},\mu_{k_i})+L \mu_{k_i}^{-1}(x^{k_i+1}-y^{k_i}), x- x^{k_i+1}\right\rangle \geq 0,\quad \forall x \in \mathcal{X}_I,\ k_i\geq K.
\end{equation}
Letting $i\to \infty$ in \eqref{fon}, by \eqref{muk0}, \eqref{par} and $\lim_{i\to \infty} x^{k_i+1}=\bar{x}$, we notice that there exists $\bar{\xi}\in \partial f (\bar{x})$ satisfying
\begin{equation}\label{opt}
\left\langle \bar{\xi}, x-\bar{x} \right\rangle \geq 0, \quad \forall x \in \mathcal{X}_I.
\end{equation}
Since $f$ is a convex function and $\mathcal{X}_I$ is a nonempty closed convex set, \eqref{opt} implies that $\bar{x}$ is a global minimizer of $f$ on $\mathcal{X}_I$. By \eqref{muxk} and $\lim_{k\to\infty}\mu_k=0$, we have
$\lim_{i\to\infty}\mu_{k_i-1}^{-1}\|x^{k_i}-x^{k_i-1}\|^2=0$, which together with the choice of $\tau_k$ in \eqref{tauk}, we deduce
$$
\lim_{i\to\infty}\left[f(x^{k_i})+\tau_{k_i}\|x^{k_i}-x^{k_i-1}\|^2\right]
=\lim_{i\to\infty}f(x^{k_i}) =f(\bar{x}).
$$
This, combined with Lemma \ref{lem:ixk}-(iv), we have $$\lim_{k\to\infty}\left[f(x^{k})+\tau_{k}\|x^{k}-x^{k-1}\|^2\right]=f(\bar{x}).$$
Assume $\hat{x}$ is an accumulation point of $\{x^k\}$ with the convergence of subsequence $\{x^{l_j}\}$. Then, it holds that
\begin{equation}\label{limf}
f(\hat{x})=\lim_{j\to\infty}f(x^{l_j})
\leq \lim_{j\to\infty}\left[f(x^{l_j})+\tau_{l_j}\|x^{l_j}-x^{l_j-1}\|^2\right]=f(\bar{x}).
\end{equation}
Since $\hat{x}\in\mathcal{X}_I$ and $\bar{x}\in\arg\min_{x\in\mathcal{X}_I} f(x)$, by \eqref{limf}, we obtain $\hat{x}\in\arg\min_{x\in\mathcal{X}_I} f(x)$.
Hence, any accumulation point of $\{x^k\}$ is a global minimizer of $f$ on $\mathcal{X}_I$. Equivalently, any accumulation point of $\{x^k\}$ is a local minimizer of problem \eqref{prob1}.
\end{proof}

Combining the construction and convergence analysis of the SFIHT algorithm, we know that the extrapolation coefficients can be considered into two cases for simplicity, i.e.,
$$
	\left\{
\begin{aligned}
&\beta_k \in \left[0,\sqrt{\left(1-\frac{1}{2k^{1-\sigma}}\right)\frac{\mu_k}{\mu_{k-1}}}\ \right]
& &\mbox{if}~ I(x^{k-1})=I(x^k)=I(\bar{x}^{k+1}),& \\
& \beta_k \in \left[0,\sqrt{\frac{L-L_{\tilde{f}}}{8L-4L_{\tilde{f}}}\frac{\mu_k}{\mu_{k-1}}}\ \right]& &\mbox{otherwise}.&
\end{aligned}\right.
$$
for which the theoretical results in this paper are still valid. It's well known that the condition of the extrapolation parameters is weaker is better and $\sup_k\beta_k=1$ is the most important. Thus, we divide the extrapolation coefficients into three cases, mainly to expand the range of extrapolation coefficients so as to greatly improve the convergence performance of the SFIHT algorithm. Though this behavior will increase the computation of the algorithm, it is surprising by Lemma \ref{lem:ixk}-(i) that the computational increase is finite.

\subsubsection{Convergence rate on objection function values}
In this subsubsection, we discuss a specific $\beta_k$ in Step 1 of the SFIHT algorithm. It not only keeps the validity for finding the local minimizers of problem \eqref{prob1}, but also obtains the convergence rate of the loss and objective function values. Based on the FISTA algorithm in \cite{Beck2009A} and the smoothing method, Bian in \cite{Bian2020smoothing} proposed a smoothing fast iterative shrinkage thresholding algorithm ($\mathrm{FISTA}_S$) for solving the constrained nonsmooth convex optimization problem. We set extrapolation coefficient $\beta_k$ in Step 1 to be the same as that in \cite{Bian2020smoothing}, namely, the following recurrence relation:
\begin{equation}\label{lossbeta}
	\left\{
	\begin{aligned}
		%\begin{split}
		&t_{k}=\frac{1+\sqrt{1+4\left(\frac{\mu_{k-1}}{\mu_{k}}\right)t_{k-1}^2}}{2},& \\
		&\beta_{k}=\frac{t_{k-1}-1}{t_{k}},&
		%\end{split}
	\end{aligned}\right.
\end{equation}
with $t_0=1$. Then, $\beta_k< \sqrt{\frac{\mu_k}{\mu_{k-1}}}$, which satisfies the condition in Step 1. Since the SFIHT algorithm always runs (3a) in Step 3 after finite iterations, by \cite[Theorem 1]{Bian2020smoothing}, we can directly get the following corollary.
\begin{corollary}\label{cor:cor1}
	Let $\{x^k\}$ be the sequence generated by the SFIHT algorithm, in which the $\beta_k$ in Step 1 is defined by \eqref{lossbeta}. Then, the following statements hold:
	\begin{enumerate}[{\rm (i)}]
		\item any accumulation point of $\{x^k\}$ is a local minimizer of problem \eqref{prob1};
		\item $\lim_{k\to\infty}F(x^k)=F_{\infty}<\infty$ exists and
		\begin{equation*}
			F(x^{k+1})- F_\infty =\left\{
			\begin{aligned}
				%\begin{split}
				&O(k^{-\sigma})& & \mbox{if}~ 0<\sigma< 1,& \\
				&O(k^{-1}\ln k)& & \mbox{if}~ \sigma= 1,& \\
				&O(k^{\sigma-2})& & \mbox{if}~ 1<\sigma< 2.& \\
				%\end{split}
			\end{aligned}\right.
		\end{equation*}
	\end{enumerate}
%	where $x^*\in \arg\min\{f(x):x\in\mathcal{X}_I\}$, $I$ is defined in \cref{the:theorem1}.
\end{corollary}

Since the $\ell_0$ norms of the iterates change only a finite number of times,  the above convergence rate also holds for loss function $f$.
From the estimation in Corollary \ref{cor:cor1}, we can see that the convergence rate on the objective function values generated by the SFIHT algorithm greatly improves the rate given by the algorithm in \cite{Bian2020a}.
\section{A case on smooth loss function}
\label{sec:3}
In this section, we discuss the case that loss function $f$ in problem \eqref{prob1} is Lipschitz continuously differentiable. In summary, throughout this section, we require $f$ in \eqref{prob1} to satisfy the following assumptions
\begin{equation*}
\left\{
\begin{aligned}
&\bullet~  f~ \mbox{is a smooth convex function on}~ \mathcal{X};\\
&\bullet~  f~ \mbox{is level bounded on}~ \mathcal{X};\\
&\bullet~ \nabla f~ \mbox{is Lipschitz continuous with Lipschitz constant}~ L_f.  \\
\end{aligned}\right.
\end{equation*}
In the context of this problem, we select a special extrapolation coefficient $\beta_k$ in the Step 1 of the SFIHT algorithm. Since the smoothing method is not needed in this case, we call it fast iterative hard thresholding (FIHT) algorithm. Please see Algorithm \ref{alg:alg2}. For the sake of brevity, we also define an approximation of $F$ at a given point $y$ as follows:
$$Q(x,y):=f(y) + \langle \nabla f(y),x-y \rangle + \frac{L}{2} \|x-y\|^2+\lambda\|x\|_0,$$
where $L>L_f$.
\begin{algorithm}
	\caption{Fast Iterative Hard Thresholding (FIHT) algorithm}
	\label{alg:alg2}
	\begin{algorithmic}
		\STATE{\textbf{Initialization:} Take $x^{1}=x^0\in \mathcal{X}$, $\alpha >0$ and $L>L_f$. Set $k=1$.}
		
		\WHILE{a termination criterion is not met,}
		
		\STATE{\begin{enumerate}
				\item [\textbf{Step 1.}] Take $\beta_k=\frac{k-1}{k+\alpha-1}$.
				
				\item [\textbf{Step 2.}] Compute
				\begin{align*}
				& y^k=x^k+\beta_k(x^k-x^{k-1}),   \\
				& \bar{x}^{k+1}\in \arg\min\{Q(x,y^k):x\in \mathcal{X}\}.
				\end{align*}
				
				\item [\textbf{Step 3.}] \textbf{(3a)} If $I(x^{k-1})=I(x^k)=I(\bar{x}^{k+1})$, let $$x^{k+1}=\bar{x}^{k+1},$$ increment $k$ by one and return to \textbf{Step 1}.
				
				\textbf{(3b)} Otherwise, choose $\beta_k \in \left[0,\sqrt{\frac{L-L_f}{4L}}\ \right],$ compute \textbf{Step 2} to obtain $\bar{x}^{k+1}$.
				
				\quad\quad\textbf{(3b-1)} If $I(x^k)=I(\bar{x}^{k+1})$, let $$x^{k+1}=\bar{x}^{k+1},$$\quad\quad increment $k$ by one and return to \textbf{Step 1}.
				
				\quad\quad\textbf{(3b-2)} Otherwise, choose $\beta_k \in \left[0,\sqrt{\frac{L-L_f}{8L-4L_f}}\ \right],$ compute \textbf{Step 2} to\\ \quad\quad obtain $\bar{x}^{k+1}$ and let $x^{k+1}=\bar{x}^{k+1}$.
				
			    Increment $k$ by one and return to \textbf{Step 1}.
			\end{enumerate}
		}
		\ENDWHILE
		
		\textbf{Output}\quad $x^k$, $\mu_k$ and $\beta_k$.
	\end{algorithmic}
\end{algorithm}

Let $\{x^k\}$ and $\{y^k\}$ be the iterates generated by the FIHT algorithm. Similarly, the subproblem in Step 2 of FIHT algorithm has a closed-form solution. Hence, the FIHT algorithm is also well-defined. The following lemma shows a lower bound property of $\{x^k\}$, which can be easily obtained by \cite[Lemma 3.2]{Wu2020accelerated}, so we omit its proof here.

\begin{lemma}\label{lem:lowb}
	The following statements hold.
	\begin{enumerate}[{\rm (i)}]
		\item When $x_j^{k}\neq0$ for some $j\in\{1,2,\ldots,n\}$, it holds that
		\begin{equation}\label{sig}
		|x_j^{k}|\geq \delta:=\min_{i=1,2,\ldots,n}\delta_i >0,
		\end{equation}
		where
		\begin{equation*}
		\delta_i= \left\{
		\begin{aligned}
		&\min \left(u_i,\sqrt{{2\lambda}/{L}}\right)& &\mbox{if}~ l_i=0,& \\
		&\min \left(-l_i,\sqrt{{2\lambda}/{L}}\right)& &\mbox{if}~ u_i=0,& \\
		&\min \left(-l_i,u_i,\sqrt{{2\lambda}/{L}}\right)& &\mbox{otherwise.}& \\
		\end{aligned}\right.
		\end{equation*}
		
		\item For every $k\in\mathbb{N}$, if $I(x^k)\neq I(x^{k+1})$, then
		$\|x^{k+1}-x^k\|\geq\delta.$
	\end{enumerate}
\end{lemma}

We begin the convergence analysis of the FIHT algorithm by defining the following important auxiliary sequence
\begin{equation*}
W(x^k,\zeta_k):=F(x^k)+\zeta_k \|x^k-x^{k-1}\|^2,
\end{equation*}
where $\zeta_k >0$. Again, for all $k\in\mathbb{N}$, we give a choice of $\zeta_k$ as follows,
\begin{equation}\label{zeta}
\zeta_k=\left\{
\begin{aligned}
&\frac{L}{4}(1+\beta_k^2)& &\mbox{if}~I(x^{k-1})=I(x^k)=I(x^{k+1}),& \\
&\frac{L-L_{f}}{8}& &otherwise.&
\end{aligned}\right.
\end{equation}
By a similar analysis to sequence $H(x^k,\mu_{k},\tau_k)$ in Section \ref{sec:2}, we can get some basic results on the sequence $W(x^k,\zeta_k)$ for the FIHT algorithm. For easy of reading, we only list them in the following lemma, but omit their proofs.
\begin{lemma}\label{lem:wk}
	The following properties are satisfied:
	\begin{enumerate}[{\rm (i)}]
		\item for every $k\in\mathbb{N}$, $x^k\in \mathcal{X}$;
		
		\item when $I(x^{k})\neq I(x^{k+1})$, we have
		\begin{equation*}
		\begin{split}
		& W(x^{k+1},\zeta_{k+1})- W(x^k,\zeta_k) \\
		& \leq\left[\zeta_{k+1}- \frac{L-L_f}{4}\right] \|x^{k+1}-x^{k}\|^2 + \left[\frac{2L - L_f}{2}\beta_k^2-\zeta_{k} \right]\|x^{k}-x^{k-1}\|^2;
		\end{split}
		\end{equation*}
		
		\item when $I(x^k)=I(x^{k+1})$, we have
		\begin{equation*}
		\begin{split}
		&W(x^{k+1},\zeta_{k+1})- W(x^k,\zeta_k) \\
		&\leq  \left[\zeta_{k+1}-\frac{L}{2}\right] \|x^{k+1}-x^k\|^2
		+ \left[\frac{L}{2}\beta_k^2-\zeta_{k}\right]\|x^{k}-x^{k-1}\|^2;
		\end{split}
		\end{equation*}
	\item when $I(x^k)=I(x^{k+1})$, $\frac{L}{2}\beta_k^2\leq \zeta_k$; otherwise, $\frac{2L-L_f}{2}\beta_k^2\leq \zeta_k$;
	
	\item $\{W(x^{k},\zeta_{k})\}$ is nonincreasing and $\lim_{k\to\infty}W(x^{k},\zeta_{k})=W_\infty<\infty$ exists.
	\end{enumerate}
\end{lemma}

Based on the above results, the next lemma shows that the iterate sequence $\{x^k\}$ generated by the FIHT algorithm is bounded and $I(x^k)$ only changes finite times.
\begin{lemma}\label{lem:l0norm}
	The following statements hold:
	\begin{enumerate}[{\rm (i)}]
		\item there exists a positive constant $R$ such that $\|x^k\|\leq R$, $\forall k\in \mathbb{N}$;
		
		\item the $\ell_0$ norm of sequence $\{x^k\}$ changes only finitely often.
	\end{enumerate}
\end{lemma}
\begin{proof}
$\mbox{(i)}.~$ By virtue of the nonincreasing of sequence $\{W(x^{k},\zeta_{k})\}$, we have
\begin{equation*}
F(x^k)\leq W(x^{k},\zeta_{k})\leq W(x^{1},\zeta_{1})=F(x^1).
\end{equation*}
Combining this and the level boundedness of $f$ on $\mathcal{X}$, it is easy to obtain the boundedness of $\{x^k\}$.

$\mbox{(ii)}.~$ Also denote $\mathcal{K}=\{k: I(x^k)\neq I(x^{k+1})\}$, and we prove the finiteness of set $\mathcal{K}$ by contradiction. Let's assume that $\mathcal{K}$ contains infinite elements.
By (ii) and (iv) of Lemma \ref{lem:wk} and $\zeta_{k+1}=\frac{L-L_f}{8}$ for $k\in\mathcal{K}$, we obtain
\begin{equation*}
	W(x^{k+1},\zeta_{k+1})-W(x^{k},\zeta_{k})\leq -\frac{L-L_f}{8}\|x^{k+1}-x^{k}\|^2,\quad \forall k\in \mathcal{K}.
\end{equation*}
Summing up the above inequality over $k\in \mathcal{K}$ and using Lemma \ref{lem:wk}-(v), we find
\begin{equation}\label{besu}
\begin{split}
\sum_{k\in \mathcal{K}}\frac{L-L_f}{8}\|x^{k+1}-x^{k}\|^2 &\leq \sum_{k\in \mathcal{K}}\left[W(x^{k},\zeta_{k})-W(x^{k+1},\zeta_{k+1})\right]\\
&\leq \sum_{k=1}^{\infty}\left[W(x^{k},\zeta_{k})-W(x^{k+1},\zeta_{k+1})\right] \\
&=W(x^{1},\zeta_{1})-W_{\infty}<\infty.
\end{split}
\end{equation}
In addition, from Lemma \ref{lem:lowb}-(ii), we have
$$
\sum_{k\in \mathcal{K}}\|x^{k+1}-x^{k}\|^2 \geq \sum_{k\in \mathcal{K}}\delta^2 =\infty,
$$
which is inconsistent with \eqref{besu}. Hence, $I(x^k)$ only changes finite times, as claimed.
\end{proof}

From the above lemma, we can easily find that the FIHT algorithm is reduced to the proximal gradient algorithm with extrapolation $\beta_k=\frac{k-1}{k+\alpha-1}$ for solving $\min_{\mathcal{X}_J} f$ after a finite number of iterations, where $J$ is a fixed index set by Lemma \ref{lem:l0norm}. Literature \cite{Attouch2016the} studied the forward-backward method with extrapolation coefficient $\frac{k-1}{k+\alpha-1}$ for solving the sum of a convex function with Lipschitz continuous gradient and a proper closed convex function. In the context of this section, the algorithm can be written as follows:
\begin{equation}\label{foba}
\left\{
\begin{aligned}
%\begin{split}
& y^k=x^k+\frac{k-1}{k+\alpha-1}(x^k-x^{k-1}), \\
& x^{k+1}= \arg\min_{x\in \Omega}\left\{\langle \nabla f(y^k),x-y^k \rangle + \frac{L}{2} \|x-y^k\|^2\right\},
%\end{split}
\end{aligned}\right.
\end{equation}
where $\Omega$ is a nonempty closed convex set of $\mathbb{R}^n$. The convergence results of this algorithm in \cite{Attouch2016the} are as follows.
\begin{lemma}\label{lem:atth}\cite{Attouch2016the}
	Let $\{x^k\}$ be a sequence generated by \eqref{foba}. When $\alpha>3$, the following statements hold:
	\begin{enumerate}[{\rm (i)}]
		\item the iterates $\{x^k\}$ converge to a global minimizer of $\min_{\Omega} f$;
		
		\item $\lim_{k\to\infty}k^2(f(x^k)-\min f)=0$ and $\lim_{k\to\infty}k\|x^{k+1}-x^k\|=0.$
	\end{enumerate}
\end{lemma}

Using the above lemma, we obtain the following significant conclusions for the proposed FIHT algorithm.
\begin{theorem}
	When $\alpha>3$, sequence $\{x^k\}$ generated by the FIHT algorithm satisfies
	\begin{enumerate}[{\rm (i)}]
		\item $\lim_{k\to\infty} x^k=x^*$, where $x^*$ is a local minimizer of problem \eqref{prob1} satisfying the lower bound property:
		\begin{equation}\label{xstart}
		|x^*_i|\geq\delta \quad \mbox{or} \quad x^*_i=0,\quad \forall i=1,2,\dots,n,
		\end{equation}
		where $\delta$ is defined as in \eqref{sig}. Moreover, $\lim_{k\to\infty}f(x^k)=f_\infty<\infty$ and $\lim_{k\to\infty}F(x^k)=F_\infty<\infty$ exist;
		
		\item $\lim_{k\to\infty}k\|x^{k+1}-x^k\|=0$;
		
		\item $\lim_{k\to\infty}k^2(f(x^k)-f_\infty)=0$ and $\lim_{k\to\infty}k^2(F(x^k)-F_\infty)=0$.
	\end{enumerate}
\end{theorem}

\begin{proof}
$\mbox{(i)}.~$ Statement (ii) of Lemma \ref{lem:l0norm} implies that there exist $\bar{K}\in\mathbb{N}$ and $J\subset \{1,2,\ldots,n\}$ such that $I(x^k)=J$, $\forall k \geq \bar{K}$. Therefore, we have
\begin{equation}\label{sek}
x^{k+1}=\arg\min_{x\in \mathcal{X}_J}\left\{f(y^k)+\langle\nabla f(y^k),x-y^k\rangle+\frac{L}{2}\|x-y^k\|^2\right\},\quad\forall k\geq \bar{K}.
\end{equation}
Observing the above relation and Lemma \ref{lem:atth}, we know $\lim_{k\to\infty}x^k=x^*
$, where $x^*$ is a global minimizer of $\min_{\mathcal{X}_J} f$, which is a local minimizer of problem \eqref{prob1}. Combining this with Lemma \ref{lem:lowb}-(i), we deduce that $x^*$ has the lower bound property in \eqref{xstart}. This together with the continuity of $f$, we obtain all the results in this statement.

$\mbox{(ii)}.~$ In view of \eqref{sek} and Lemma \ref{lem:atth}-(ii), we have the estimation in item (ii).

$\mbox{(iii)}.~$ From Lemma \ref{lem:l0norm}-(ii) and \eqref{xstart}, there exists $\bar{K}\in\mathbb{N}$ such that $$\|x^k\|_0=\|x^*\|_0,\quad\forall k\geq \bar{K}.$$
Further, it holds that
\begin{equation}\label{loss}
f(x^k)-f(x^*)=F(x^k)-F(x^*),\quad \forall k\geq\bar{K}.
\end{equation}
By \eqref{sek} and Lemma \ref{lem:atth}-(ii), we obtain the estimation in (iii).
\end{proof}

 If the extrapolation coefficients are chosen below some threshold, the authors in \cite{Wen2017Linear} proved that the iterates and the function value sequence of the proximal gradient algorithm with extrapolation are R-linear convergent when the objective function satisfies the error bound condition. It's worth emphasizing that when loss function $f$ satisfies the error bound condition and $\beta_k$ in the FIHT algorithm satisfies $\sup_k\beta_k<1$, the R-linear convergence of sequence $\{x^k\}$ and objective function values can also be obtained for the FIHT algorithm. For easy reading, the convergence results obtained in this paper are summarized in Table \ref{tab:tab1}.
  \renewcommand\arraystretch{1.72}
 \begin{table}[htbp]%\label{tab11}
 	{\footnotesize
 		\caption{Summary of convergence of the SFIHT algorithm}  \label{tab:tab1}
 		\begin{center}
 			\begin{tabular}{|c|c|c|c|c|c|} \hline
 			$f$ & $\beta_k$ in Step 1 & $\sigma$ & local minimizer & lower bound &  $\{F(x^k)-F_{\infty}\}$  \\   \hline	
\multirow{5}{*}{nonsmooth} & $\left[0, \sqrt{\frac{\mu_{k}}{\mu_{k-1}}} \right)$ & $\left(0, 2\right)$ & --  &  -- & --  \\ %\hline
 		\cline{2-6}
 		& \eqref{be1} &  $\left[\frac{1}{2}, 1\right]$ & $\checkmark$ &  -- & --\\ %\hline
 				\cline{2-6}
 		&\multirow{3}{*}{\eqref{lossbeta}}  &  (0, 1) & $\checkmark$ & --  & $O(k^{-\sigma})$  \\ %\hline
 			\cline{3-6}
    	&  &  1 & $\checkmark$ & --  &  $O(k^{-1}\ln k)$  \\ %\hline
 	    	\cline{3-6}
 		&  &  (1, 2) & $\checkmark$ & --   &  $O(k^{\sigma-2})$  \\
 		\hline
 smooth & $\frac{k-1}{k+\alpha-1}$ & -- & $\checkmark$ & $\checkmark$ & $o(k^{-2})$  \\  \hline
 			\end{tabular}
 		\end{center}
 	}
 \end{table}

\section{Experimental results}
\label{sec:4}
The aim of this section is to verify the theoretical results and performance of the proposed two algorithms by some numerical experiments. The SFIHT algorithm without extrapolation is called the smoothing iterative hard thresholding (SIHT) algorithm in this paper. Example 4.1 and Example 4.2 are an under-determined linear regression problem and an over-determined censored regression problem, respectively. The purpose of Example 4.1 and Example 4.2 is to illustrate the ability of the SFIHT algorithm for solving the problem, and to compare the good performance of the SFIHT algorithm with respect to the SIHT algorithm. In Example 4.3, we use the FIHT algorithm to solve the under-determined $\ell_0$ regularized least squares problem. At the same time, we compare the performance of FIHT algorithm and IHT algorithm in Example 4.3. For different problems, we choose appropriate equilibrium parameter $\lambda$ to adjust the data fitting and sparsity. The CPU time (in seconds) reported here doesn't include the time of data initialization.

The numerical experiments are performed in Python 3.7.0 on a 64-bit Lenovo PC with an Intel(R) Core(TM) i7-10710U CPU @1.10GHz 1.61GHz and 16GB RAM.

For any given $\epsilon >0$, we call $x^\epsilon$ an $\epsilon$ local minimizer of problem \eqref{prob1}, if it holds
$$\|[\nabla \tilde{f}(x^\epsilon, \mu)]_{I(x^\epsilon)^c}\|_{\infty}\leq\epsilon\quad\mbox{and}\quad\mu\leq\epsilon,$$
where $I(x^\epsilon)^c=\{i: x_i^\epsilon\neq 0\}$. We stop the proposed algorithm if $x^k$ is an $\epsilon$ local minimizer of problem \eqref{prob1} or the number of iteration $k$ exceeds $15000$. Set some fixed parameters $\mu_0=0.7$, $L=2L_{\tilde{f}}$ and $\alpha=4$ throughout the numerical experiments.

\textbf{Example 4.1}\label{exa:4.1}
	We consider the following $\ell_0$ regularized linear regression problem:
	\begin{equation}\label{ell}
	\min_{\mathbf{-1}\leq x \leq \mathbf{1}} F(x):=\|Ax-b\|_1 + \lambda \|x\|_0,
	\end{equation}
	where $A\in \mathbb{R}^{m\times n}$ with $m<n$, $b\in \mathbb{R}^m$. We choose a smoothing function of the $\ell_1$ loss function as below, and it satisfies the conditions in Definition \ref{def:def1},
	\begin{equation}\label{theta_tilde}
	\begin{aligned}
	\tilde{f}(x,\mu)=\sum_{i=1}^m\tilde{\theta}(A_i x-b_i, \mu)
	\end{aligned}
	\quad\mbox{with}\quad
	\tilde{\theta}(z,\mu)=\left\{
	\begin{aligned}
	&|z|  &\mbox{if}~ |z|>\mu, \\
	&\frac{z^2}{2\mu}+\frac{\mu}{2}  &\mbox{if}~ |z|\leq \mu.
	\end{aligned}
	\right.
	\end{equation}	
Three choices of $\beta_k$ in Step 1 and Step 3 of the SFIHT algorithm are set as follows: $\beta_k=\frac{k-1}{k+\alpha-1}\sqrt{\left(1-\frac{1}{2k^{1-\sigma}}\right)\frac{\mu_k}{\mu_{k-1}}}$ (Step 1), $\beta_k=\sqrt{\frac{L-L_{\tilde{f}}}{4L}\frac{\mu_k}{\mu_{k-1}}}$ (Step 3b) and $\beta_k =\sqrt{\frac{L-L_{\tilde{f}}}{8L-4L_{\tilde{f}}}\frac{\mu_k}{\mu_{k-1}}}$ (Step 3b-2).  Denote $s$ the $\ell_0$ norm of true solution $x^\star$, i.e., $\|x^\star\|_0=s$. For positive integers $m$, $n$ and $s$, the data is generated as follows:
\begin{center}
	\tt{\text{$\bar{x}=\mbox{zeros}(n,1);\ \bar{n}=\mbox{randperm}(n);\ \bar{x}(\bar{n}(1:s))=\mbox{randn}(s,1);\ a=0.005;$}}\\
	\tt{\text{$x^{\star}=(\mbox{median}([\bar{x}' ;l' ;u']))';\ A=\mbox{orth}(\mbox{\tt{randn}}(m,n)')';\ b=A\ast x^\star+ a \ast\mbox{randn}(m,1).$}}
\end{center}

Set $\epsilon=10^{-3}$, $L_{\tilde{f}}=\|A^{\mathrm{T}}A\|$, $\sigma=0.95$ and $x^0=\mbox{zeros}(n,1)$.
We randomly generate $A$, $b$ and $x^\star$ with $(m, n) = (300, 1000)$ and $(m, n)=(500, 5000)$. %Corresponding results are plotted in Fig. \ref{fig:fig4.1} and Fig. \ref{fig:fig4.2}.
Fig. \ref{fig1.a} shows that the support set of sequences $\{x^k\}$ generated by the SFIHT and SIHT algorithm only change finite times and are convergent. Observing Fig. \ref{fig1.b}, we find that compared with the SIHT algorithm, the SFIHT algorithm can find a better solution with fewer iterations. From Fig. \ref{fig:fig4.1}, we see that the convergence rate of the SFIHT algorithm is faster than the SIHT algorithm, and the sparsity of the solution obtained by the SFIHT algorithm is also closer to the true solution $x^\star$. For three different stopping criterions, we record the CPU time and iterations of the two algorithms in Table \ref{tab:tab2}. It's clear that the computational cost of the SFIHT algorithm is much less than that of the SIHT algorithm. The two algorithms find $\epsilon$ local minimizer with $\epsilon=10^{-4}$ with the same iterations, because $\mu_k$ doesn't meet the termination condition when $\|[\nabla \tilde{f}(x^k, \mu_k)]_{I(x^k)^c}\|_{\infty}\leq\epsilon$ holds.

\begin{figure}[htbp]
	\centerline{
		% Use the relevant command to insert your figure file.
		% For example, with the graphicx package use
		\subfigure[]{\label{fig1.a}\includegraphics[width=0.55\textwidth]{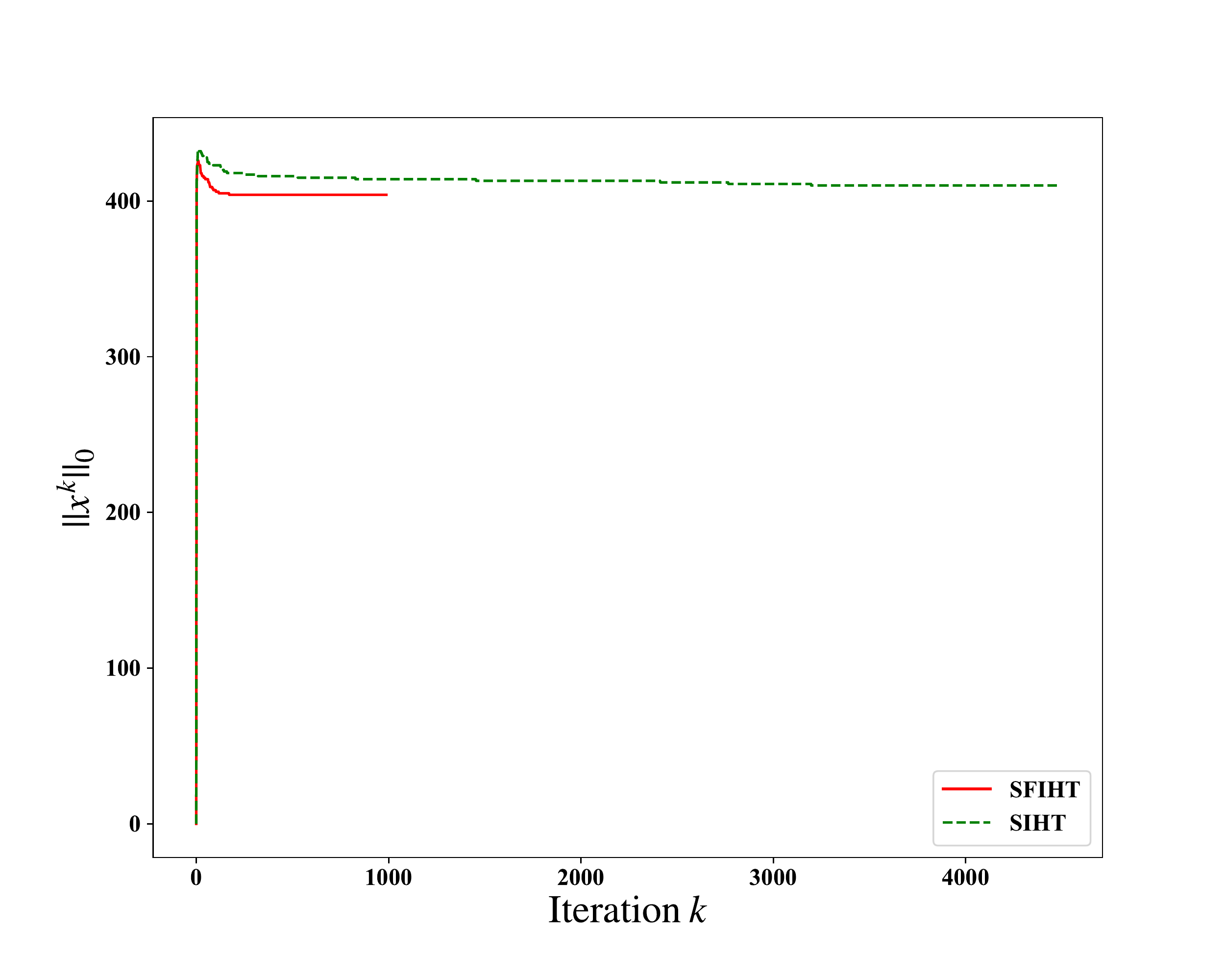}}
	    \subfigure[]{\label{fig1.b}\includegraphics[width=0.55\textwidth]{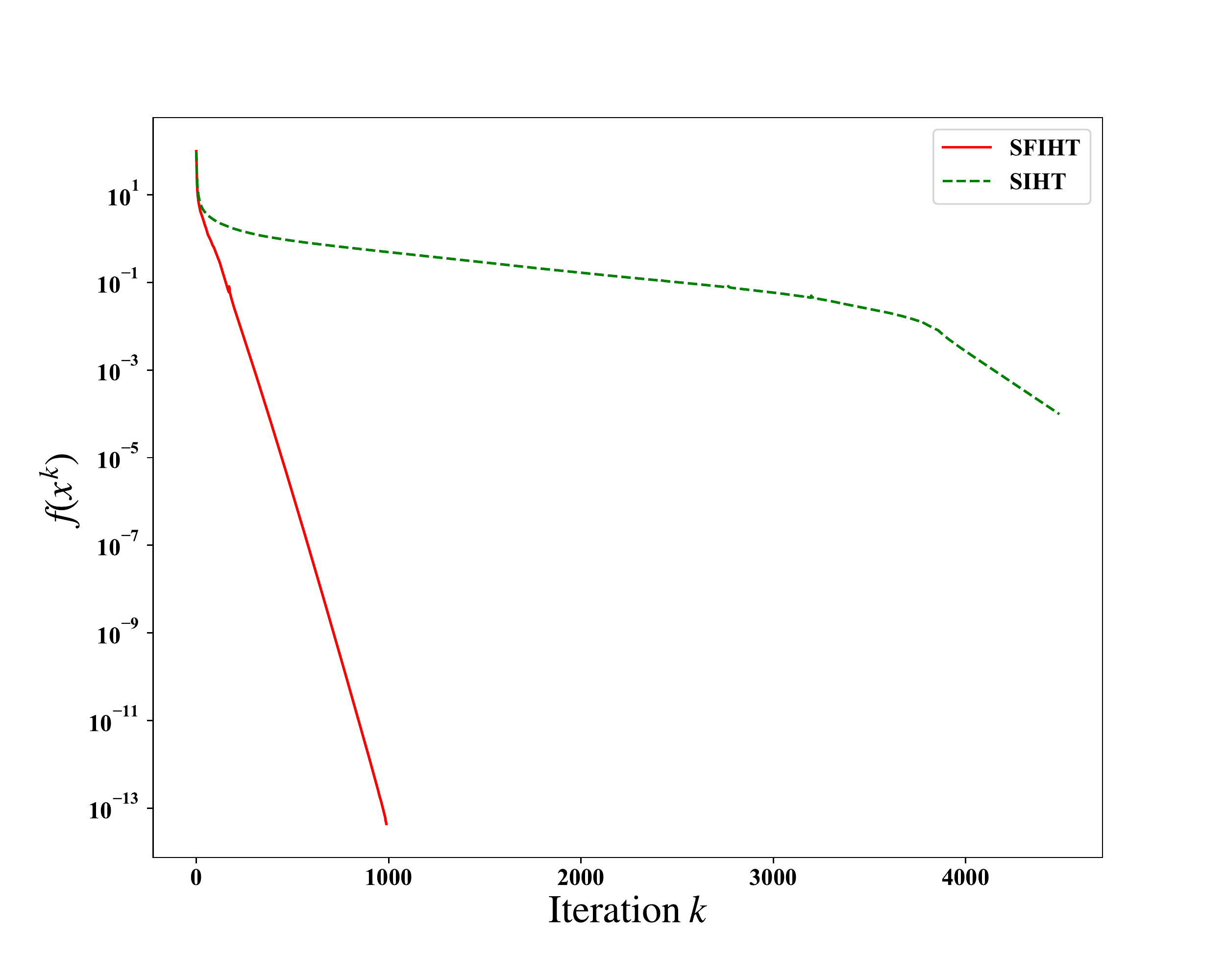}}
		% figure caption is below the figure
	}
	\caption{Convergence of cardinality and loss function values for Example 4.1 with $m=300$, $n=1000$ and $s=400.$}
	\label{fig:fig4.1}       % Give a unique label
\end{figure}

\begin{table}[htbp]
	{\footnotesize
		\caption{Computational cost for Example 4.1 with different stopping criterions when $m=300$, $n=1000$ and $s=400.$}  \label{tab:tab2}
		\begin{center}
			\begin{tabular}{|c|c|c|c|c|c|c|c|} \hline
				$\epsilon$ & \multicolumn{2}{|c|}{$10^{-2}$} & \multicolumn{2}{|c|}{$10^{-3}$} & \multicolumn{2}{|c|}{$10^{-4}$}  \\   \hline
			Algorithm  & SFIHT &  SIHT & SFIHT &  SIHT  & SFIHT &  SIHT  \\ \hline
		Time & $\mathbf{0.168}$ & $23.388$ & $\mathbf{1.644}$ & $27.108$ & $\mathbf{192.258}$ & $209.628$ \\  \hline
		Iterations & $\mathbf{217}$& $4150$ & $\mathbf{988}$ & $4487$ & $11155$ & $11155$ \\  \hline
			\end{tabular}
		\end{center}
	}
\end{table}
When $m=500$ and $n=5000$, we generate the problem data with three different sparsity levels, which are $20 \%$, $30 \%$ and $40 \%$. The results drawn in Fig. \ref{fig:fig4.2} show that the SFIHT algorithm substantially outperforms the SIHT algorithm in terms of the solution quality both on the loss function value and the cardinality.

\begin{figure}[htbp]
	\centerline{
		% Use the relevant command to insert your figure file.
		% For example, with the graphicx package use
		\subfigure[]{\label{fig2.a}\includegraphics[width=0.55\textwidth]{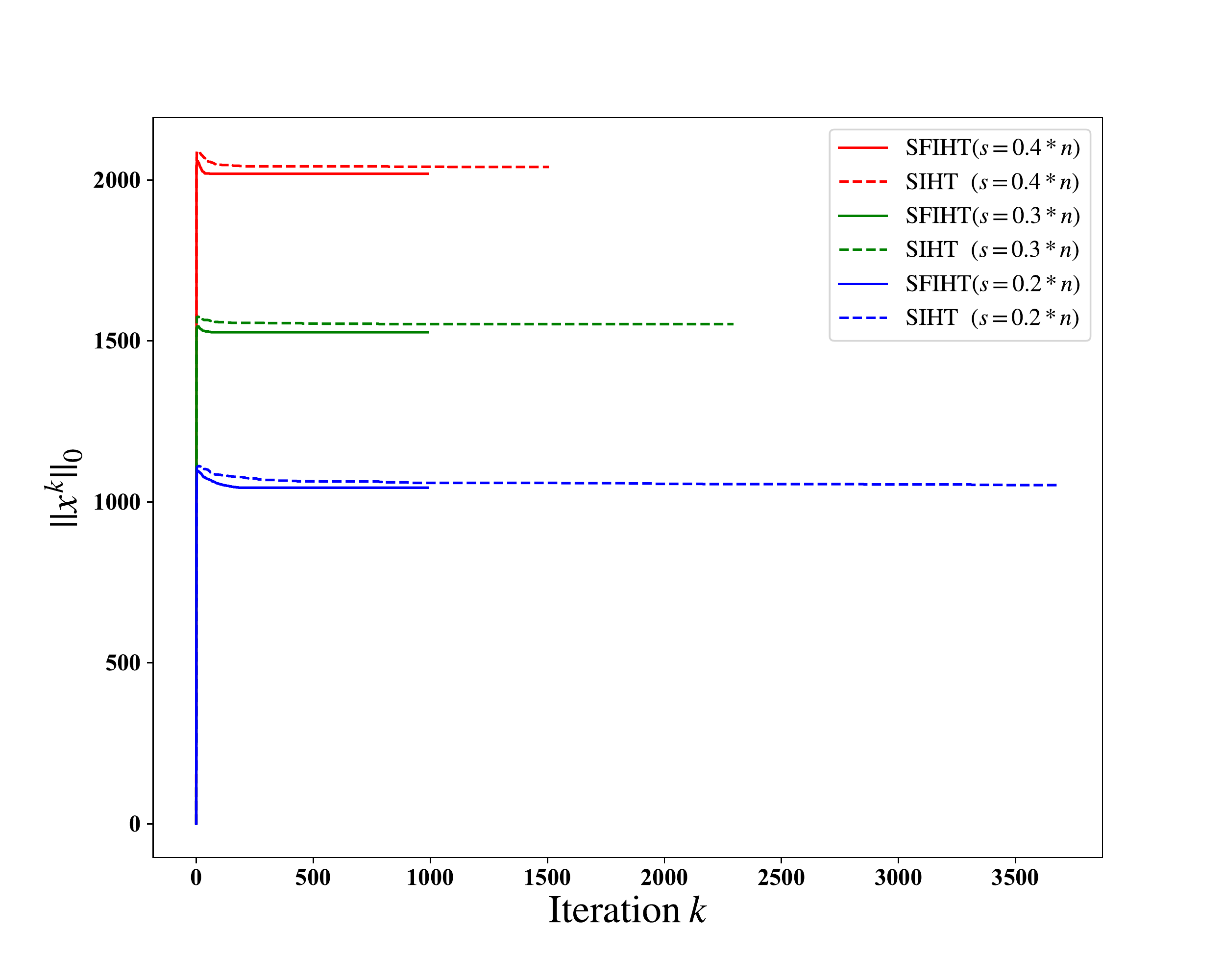}}
		\subfigure[]{\label{fig2.b}\includegraphics[width=0.55\textwidth]{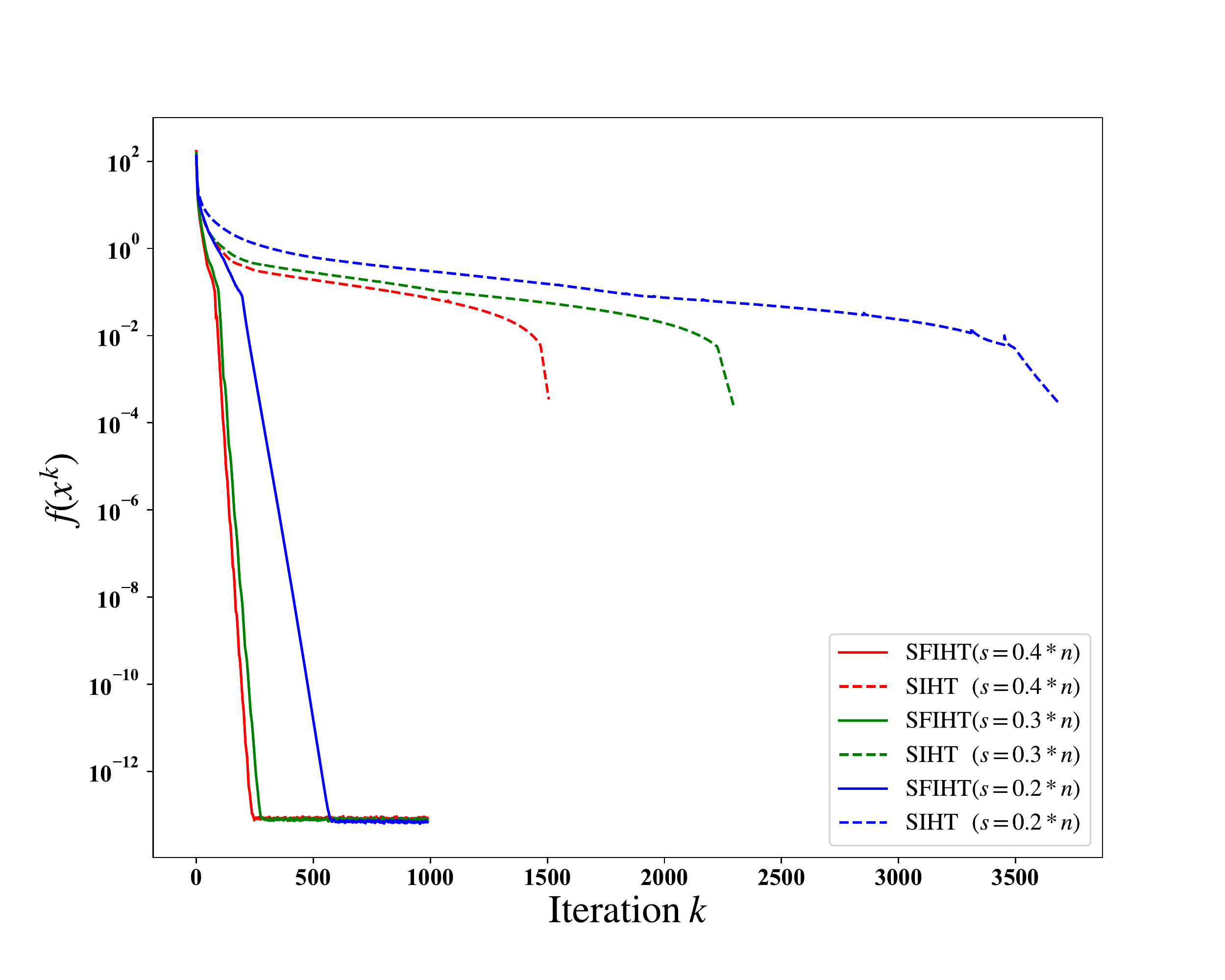}}
		% figure caption is below the figure
	}
	\caption{Convergence of cardinality and loss function values for Example 4.1 with $m=500$, $n=5000$ and different sparsity levels.}
	\label{fig:fig4.2}       % Give a unique label
\end{figure}

\textbf{Example 4.2}\label{exa:4.2}
We consider the following $\ell_0$ regularized censored regression problem:
\begin{equation*}\label{max}
	\min_{\mathbf{0}\leq x \leq \mathbf{1}} F(x):=\frac{1}{m}\|\max\{A x, 0\}-b\|_1 + \lambda \|x\|_0,
\end{equation*}
where $A\in \mathbb{R}^{m\times n}$ with $m>n$ and $b\in \mathbb{R}^m$. A smoothing function satisfying Definition \ref{def:def1} of the above loss function is defined by
\begin{equation*}
	\begin{aligned}
		\tilde{f}(x,\mu)=\frac{1}{m}\sum_{i=1}^m\tilde{\theta}(\tilde{\phi}(A_i x, \mu)-b_i, \mu)
	\end{aligned}
	\quad\mbox{with}\quad
	\tilde{\phi}(s,\mu)=\left\{
	\begin{aligned}
		&\max\{s,0\} &\mbox{if}~ |s|>\mu, \\
		&\frac{(s+\mu)^2}{4\mu} &\mbox{if}~ |s|\leq \mu.
	\end{aligned}
	\right.
\end{equation*}
 Set $\epsilon=10^{-2}$,  $L_{\tilde{f}}=\frac{3}{2m}\|A^{\mathrm{T}}A\|$, $\sigma=0.7$ and $x^0=0.1*\mbox{ones}(n,1)$. $\beta_k$ in Step 1 of the SFIHT algorithm is the same as that in \eqref{lossbeta}, and the others are the same as in Example 4.1. We randomly generate the problem data as follows:
\begin{center}
	\tt{\text{$A=\mbox{randn}(m,n);\ \bar{n}=\mbox{randperm}(n);\ x^\star=\mbox{zeros}(n,1);$}}\\
	\tt{	\text{$x^\star(\bar{n}(1:s))=\mbox{unifrnd}(0.1,1,[s,1]);\ b=\max(A*x^\star+0.01*\mbox{randn}(m,1),0).$}}
\end{center}
In this example, we run numerical experiments with $(m, n, s)=(1000, 200, 60)$ and
$(m, n, s)=(2000, 400, 80)$. Results recorded in Fig. \ref{fig:fig4.3} and Fig. \ref{fig:fig4.4} show that the SFIHT algorithm performs much better than the SIHT algorithm in terms of both the $\ell_0$ norms and loss function values. Similarly, the SFIHT algorithm needs much less iterations to get an $\epsilon$ local minimizer of problem \eqref{prob1} than the SIHT algorithm.

\begin{figure}[htbp]
	\centerline{
		% Use the relevant command to insert your figure file.
		% For example, with the graphicx package use
		\subfigure[]{\label{fig3.a}\includegraphics[width=0.55\textwidth]{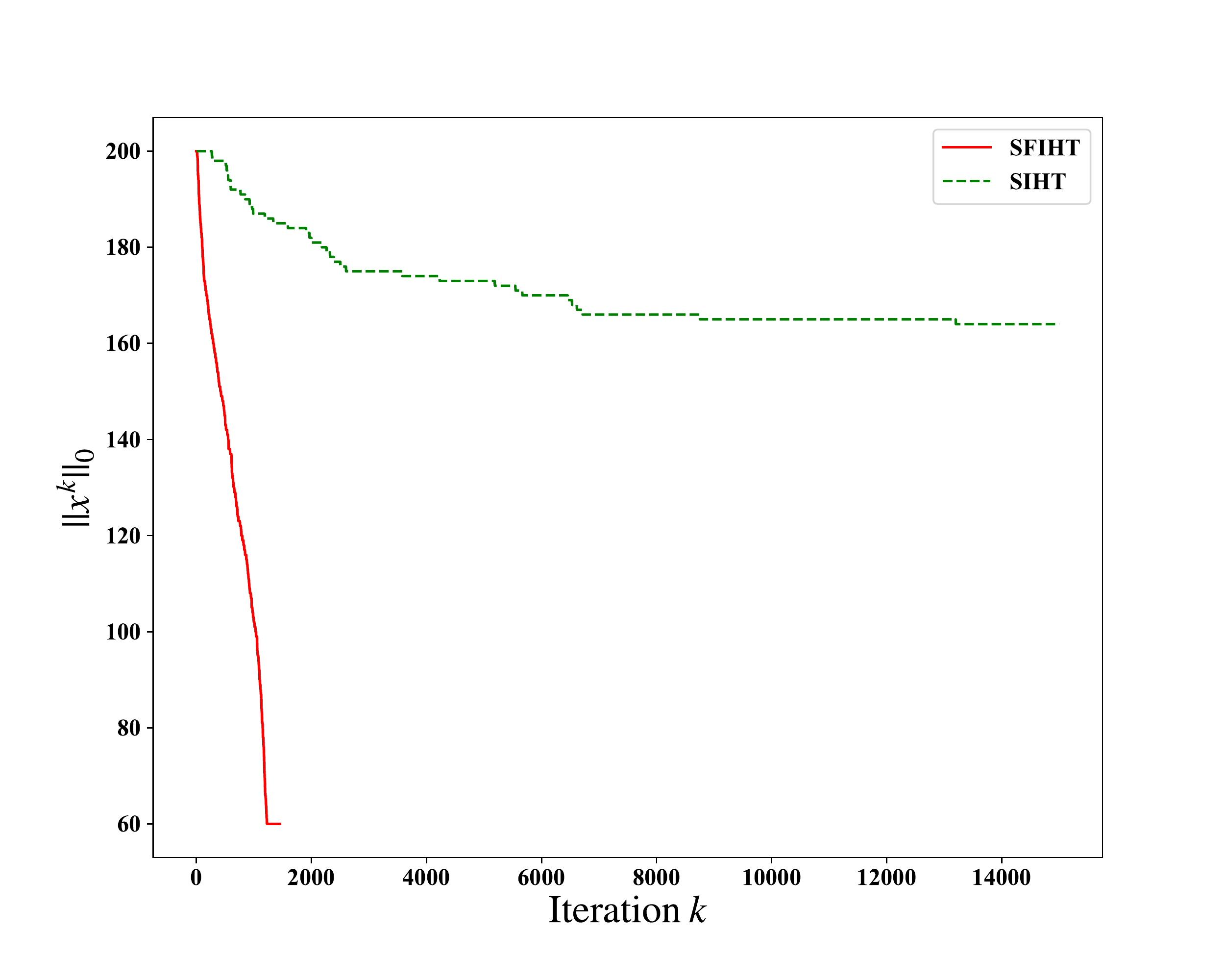}}
		\subfigure[]{\label{fig3.b}\includegraphics[width=0.55\textwidth]{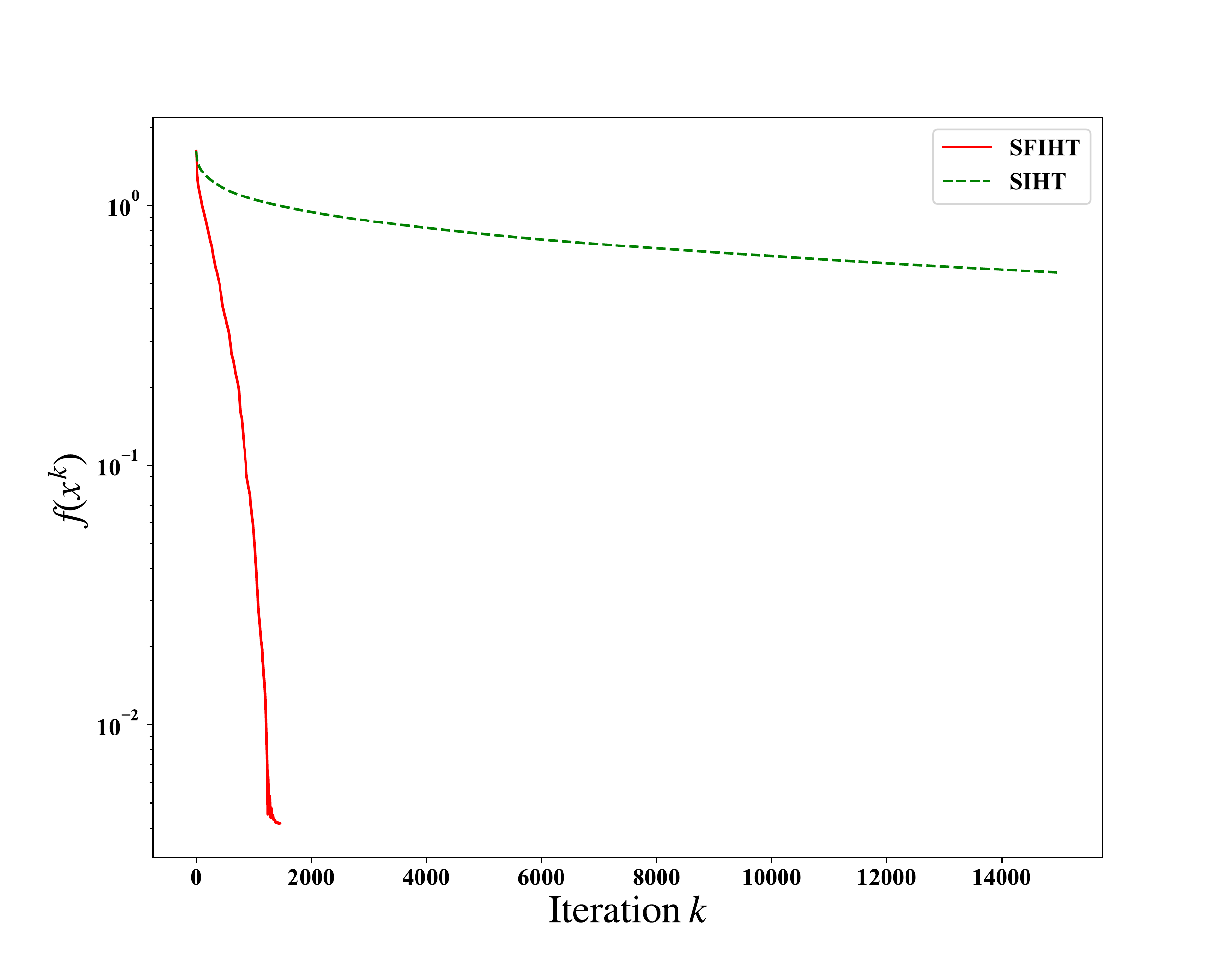}}
		% figure caption is below the figure
	}
	\caption{Convergence of cardinality and loss function values for Example 4.2 with $m=1000$, $n=200$ and $s=60$.}
	\label{fig:fig4.3}       % Give a unique label
\end{figure}

\begin{figure}[htbp]
	\centerline{
		% Use the relevant command to insert your figure file.
		% For example, with the graphicx package use
		\subfigure[]{\label{fig4.a}\includegraphics[width=0.55\textwidth]{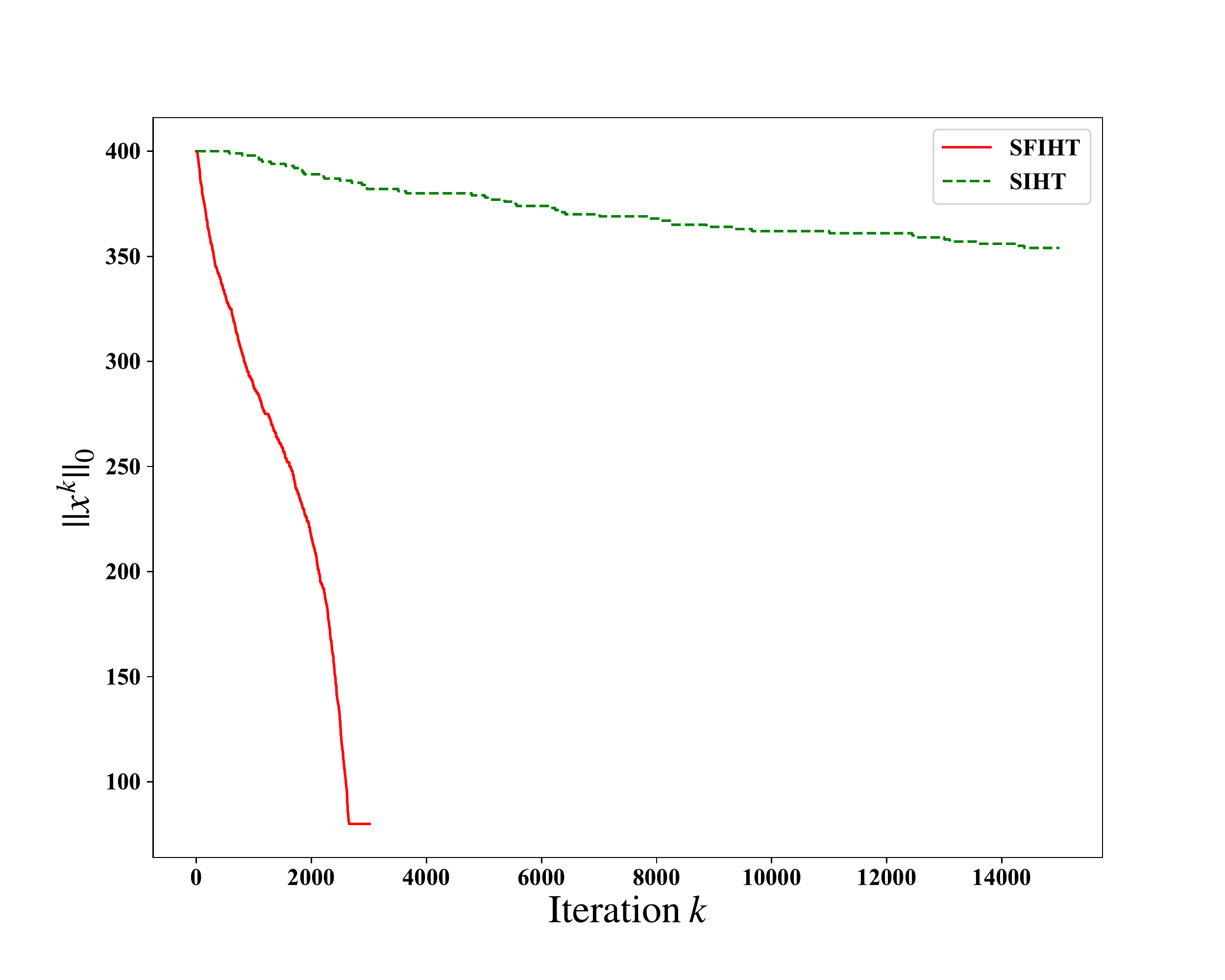}}
		\subfigure[]{\label{fig4.b}\includegraphics[width=0.55\textwidth]{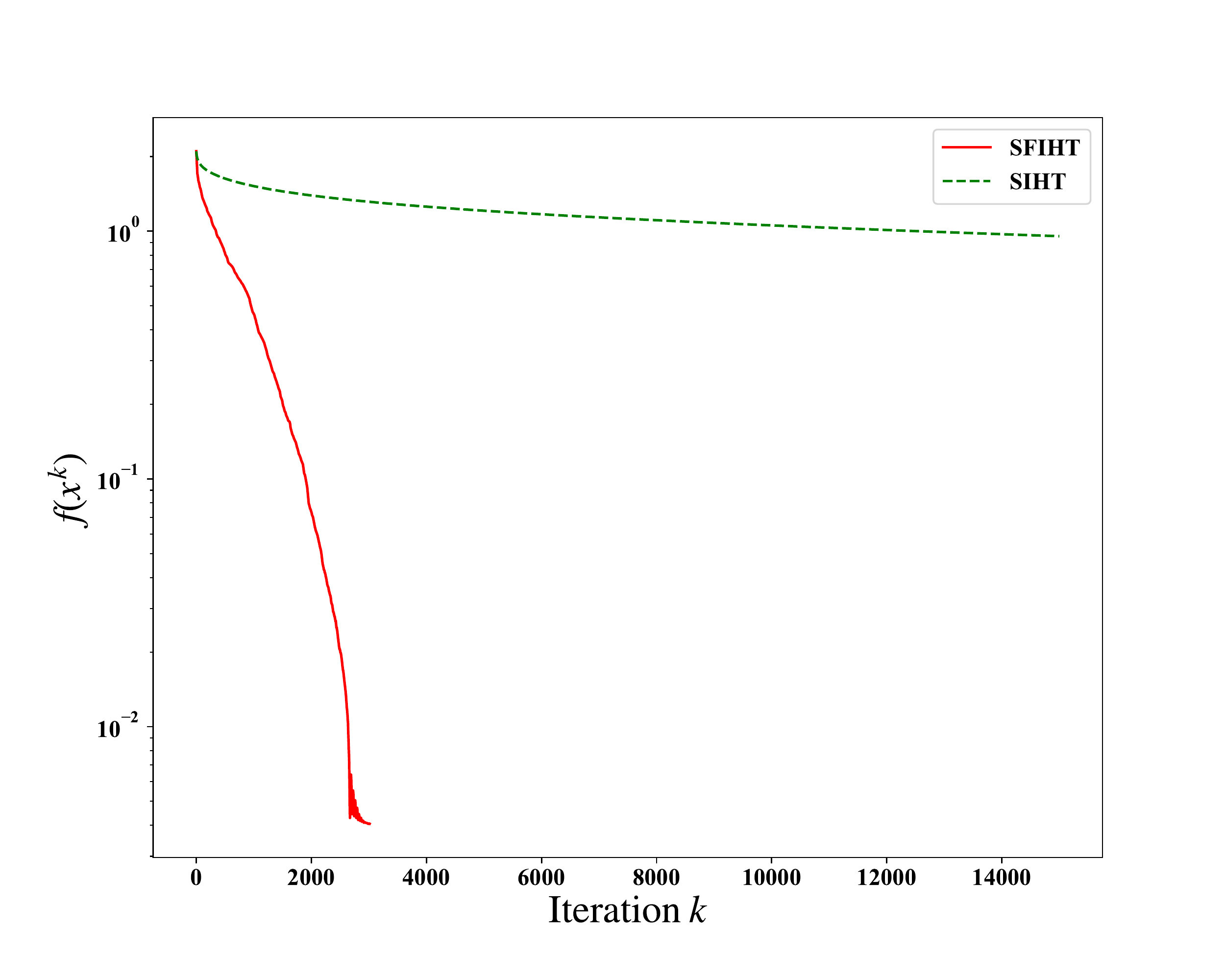}}
		% figure caption is below the figure
	}
	\caption{Convergence of cardinality and loss function values for Example 4.2 with $m=2000$, $n=400$ and $s=80$.}
	\label{fig:fig4.4}       % Give a unique label
\end{figure}

\textbf{Example 4.3}\label{exa:4.3}
We consider the following $\ell_0$ regularized least squares problem:
\begin{equation}\label{l2}
\min_{\mathbf{0}\leq x \leq \mathbf{5}} F(x):= \frac{1}{2}\|Ax-b\|^2+\lambda \|x\|_0
\end{equation}
where $A\in \mathbb{R}^{m\times n}$ with $m<n$ and $b\in \mathbb{R}^m$.

We set $\epsilon=10^{-4}$, $L_f=\|A^{\mathrm{T}}A\|$ and $x^0=\mbox{zeros}(n,1)$. We take the extrapolation coefficients
$\beta_k=\sqrt{\frac{k}{k+1}\frac{L-L_f}{4L}}$ (Step 3b) and $\beta_k =\sqrt{\frac{k}{k+1}\frac{L-L_f}{8L-4L_f}}$ (Step 3b-2) in the FIHT algorithm.
 For two cases of $(m, n)$ and three cases of $s$, we randomly generate the problem data as them in Example 4.1. When $(m, n)=(300, 1000)$, for the different choices of $s$, the cardinalities and loss function values versus iteration $k$ are plotted in Fig. \ref{fig:fig4.5}. From this figure, we can see that the total iterations of the FIHT algorithm is much smaller than that of the IHT algorithm, and both the cardinalities and loss function values of the final output iterate obtained by the FIHT algorithm are more accurate. Fig. \ref{fig:fig4.6} illustrates the convergence rate of the FIHT algorithm is also faster than that of the IHT algorithm for solving problem \eqref{l2} when the dimension of the problems is larger. It's also worth emphasizing that the loss function value at iterative point generated by SFIHT algorithm is smaller for all iterations. For different values of $\epsilon$, the CPU time and iterations for obtaining an $\epsilon$ local minimizer by the FIHT algorithm and the IHT algorithm are recorded in Table \ref{tab:tab3}. From the above results, We can observe that the FIHT algorithm performs much better than the IHT algorithm.

\begin{figure}[htbp]
	\centerline{
		% Use the relevant command to insert your figure file.
		% For example, with the graphicx package use
		\subfigure[]{\label{fig5.a}\includegraphics[width=0.55\textwidth]{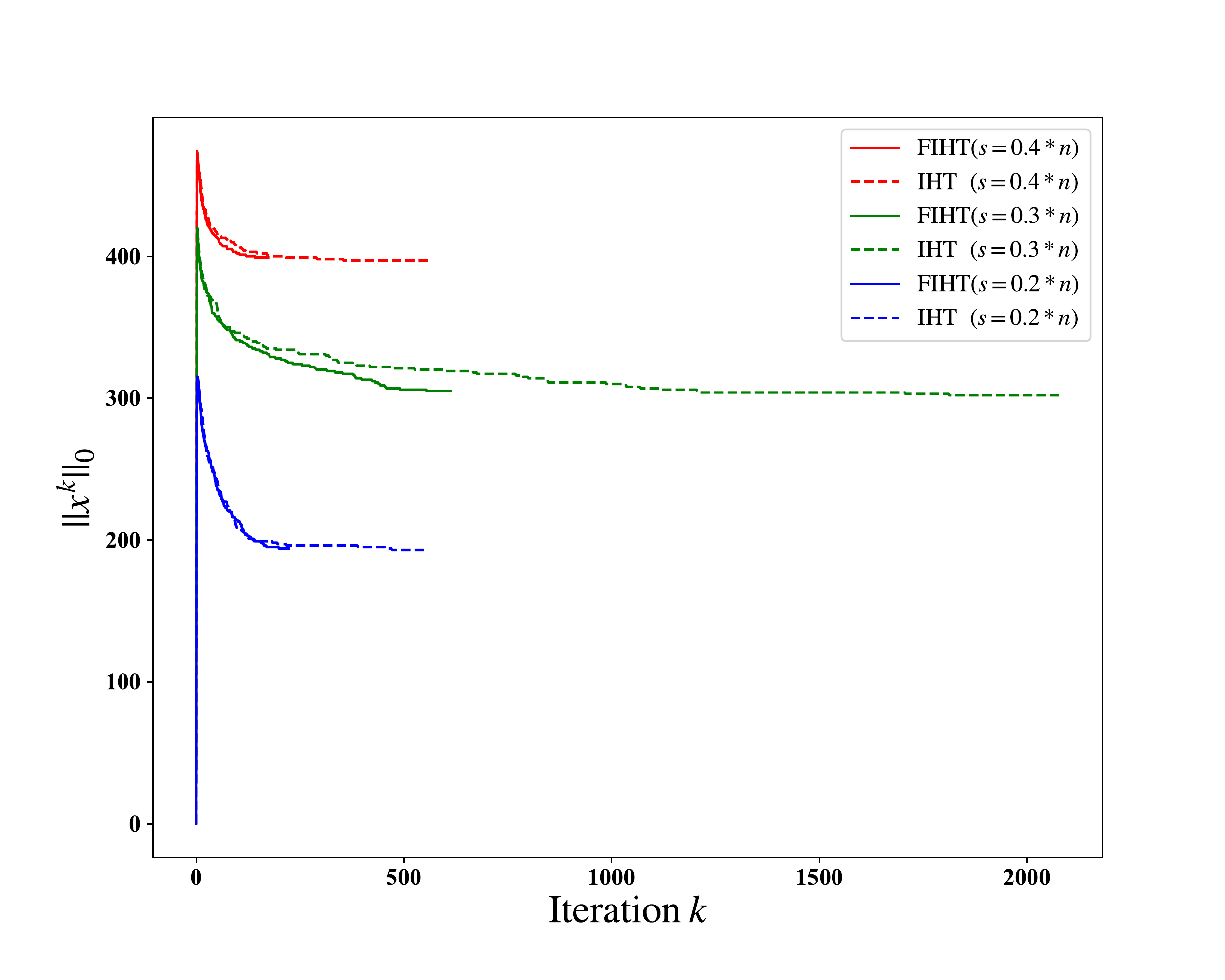}}
		\subfigure[]{\label{fig5.b}\includegraphics[width=0.55\textwidth]{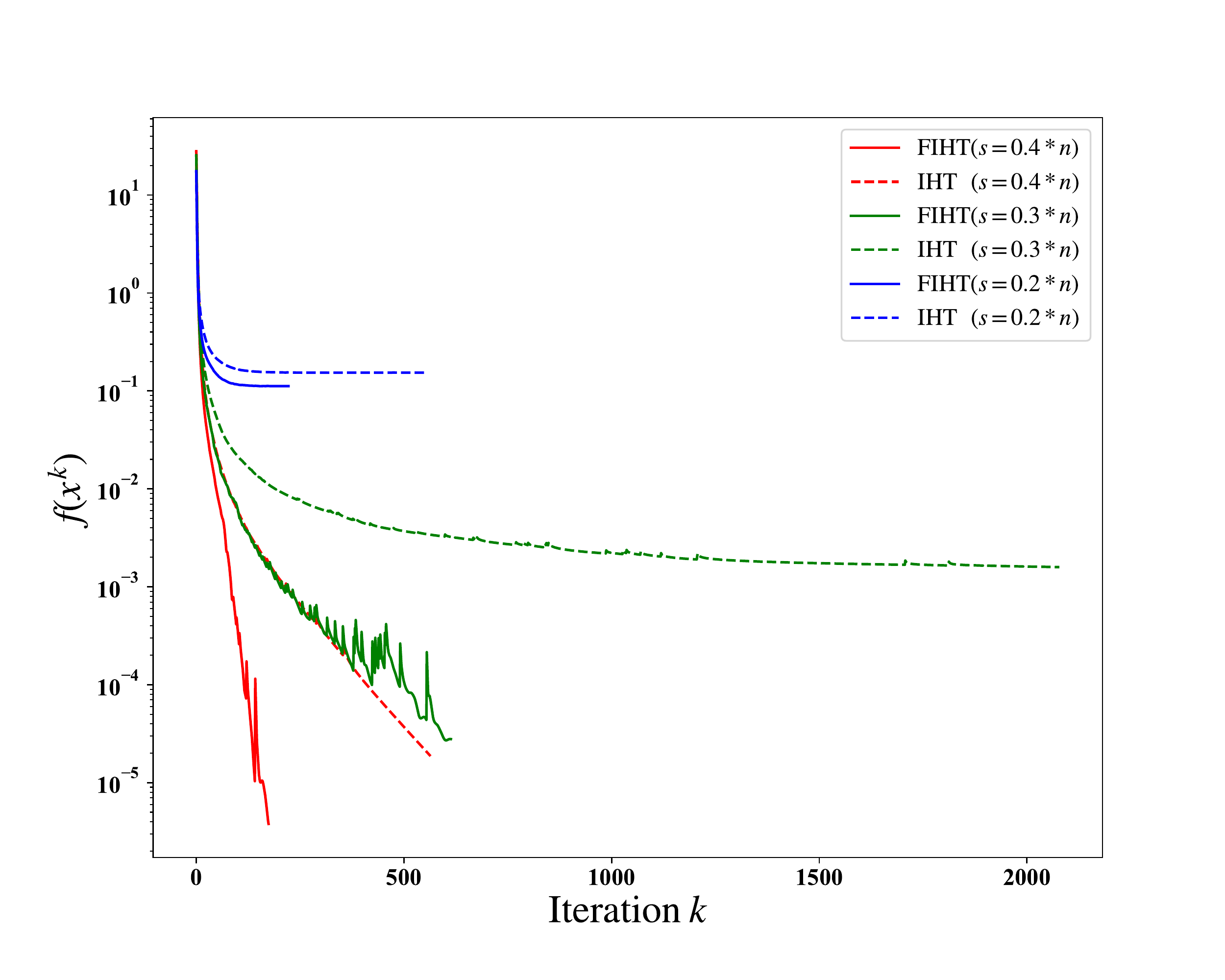}}
		% figure caption is below the figure
	}
	\caption{Convergence of cardinality and loss function values for Example 4.3 with $m=300$, $n=1000$ and different sparsity levels.}
	\label{fig:fig4.5}       % Give a unique label
\end{figure}

\begin{figure}[htbp]
	\centerline{
		% Use the relevant command to insert your figure file.
		% For example, with the graphicx package use
		\subfigure[]{\label{fig6.a}\includegraphics[width=0.55\textwidth]{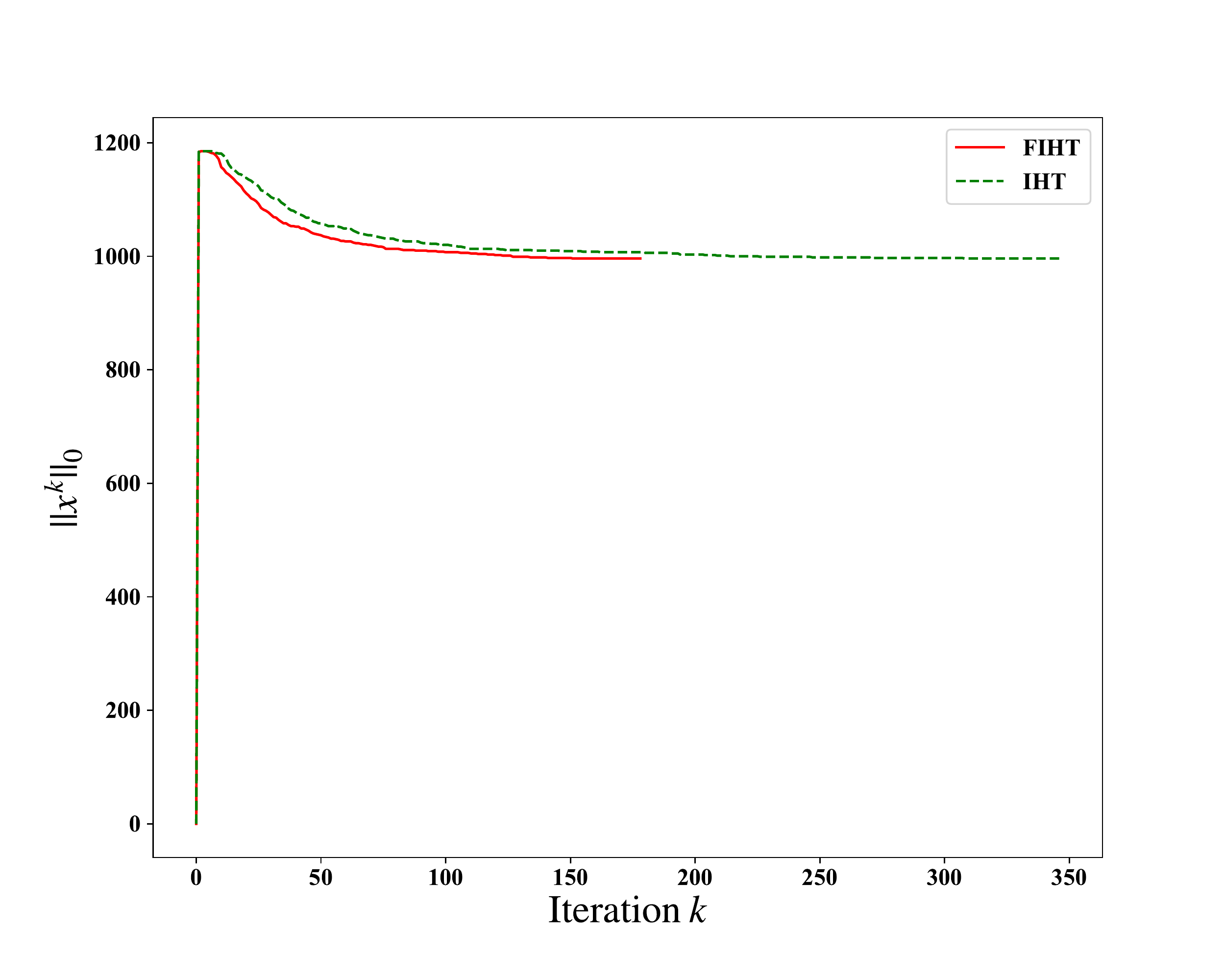}}
		\subfigure[]{\label{fig6.b}\includegraphics[width=0.55\textwidth]{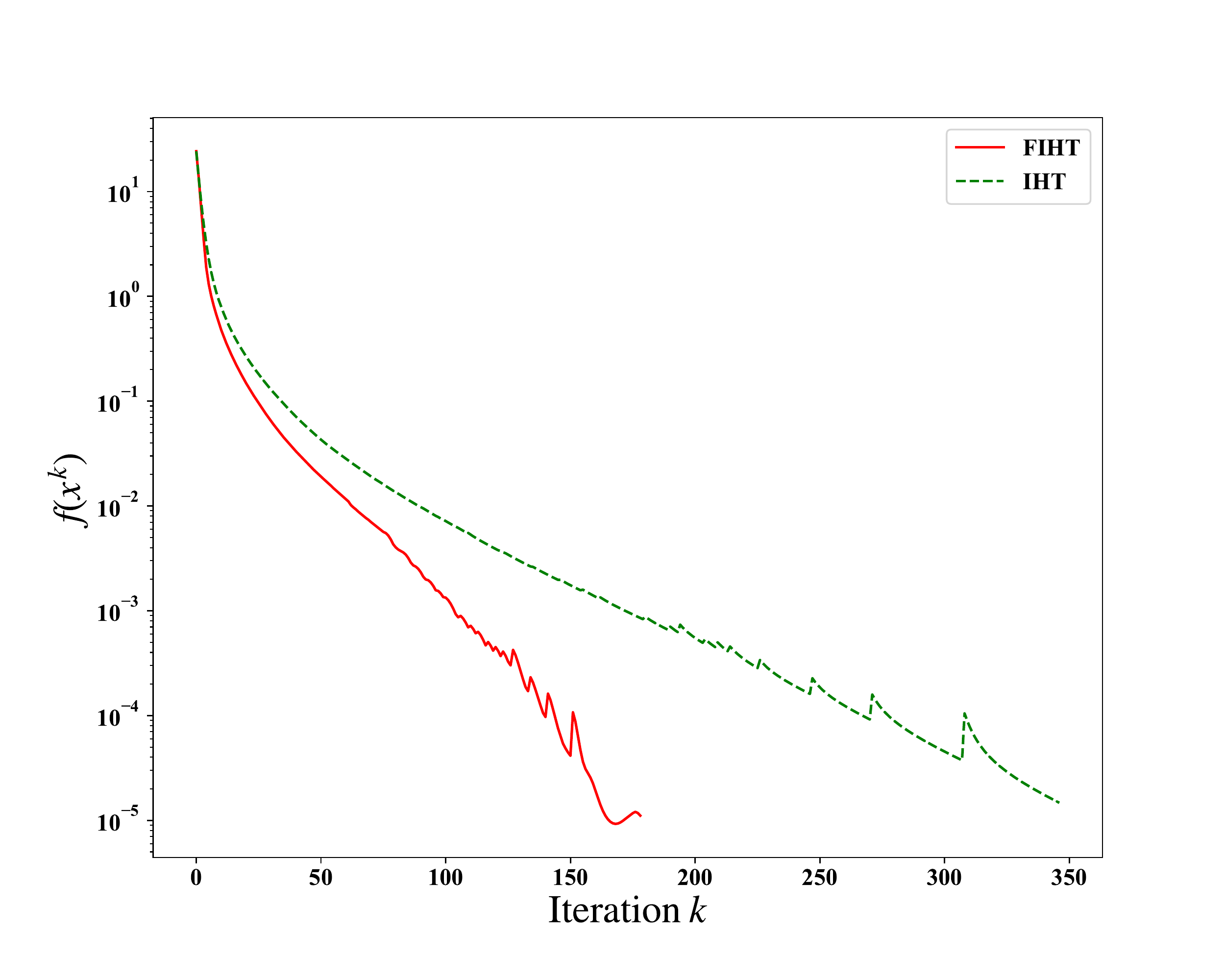}}
		% figure caption is below the figure
	}
	\caption{Convergence of cardinality and loss function values for Example 4.3 with $m=500$, $n=5000$ and $s=1000$.}
	\label{fig:fig4.6}       % Give a unique label
\end{figure}

\begin{table}[htbp]
	{\footnotesize
		\caption{Computational cost for Example 4.3 with different stopping criterions when $m=500$, $n=5000$ and $s=1000.$}  \label{tab:tab3}
		\begin{center}
			\begin{tabular}{|c|c|c|c|c|c|c|c|c|} \hline
				$\epsilon$ & \multicolumn{2}{|c|}{$10^{-2}$} & \multicolumn{2}{|c|}{$10^{-3}$} & \multicolumn{2}{|c|}{$10^{-4}$} & \multicolumn{2}{|c|}{$10^{-5}$} \\   \hline
				Algorithm  & FIHT &  IHT & FIHT &  IHT  & FIHT &  IHT & FIHT &  IHT  \\ \hline
				Time & $0.180$ & $\mathbf{0.144}$ & $\mathbf{0.696}$ & $0.738$ & $\mathbf{1.302}$ & $2.322$  &  $\mathbf{2.280}$ & $4.41$ \\  \hline
				Iterations & $\mathbf{26}$& $37$ & $\mathbf{92}$ & $144$ & $\mathbf{178}$ & $346$ & $\mathbf{284}$ & $542$ \\  \hline
			\end{tabular}
		\end{center}
	}
\end{table}

\section{Conclusions}
\label{sec:conclusions}
The main contribution of this paper is to propose an effective and fast algorithm for solving the constrained cardinality penalty problem with a continuous convex loss function, and analyze its convergence properties. We first use a parametric smoothing approximation of the loss function to generate a cardinality penalty problem with smooth loss function. Then, the iterative hard thresholding algorithm with extrapolation is used, in which the smoothing parameter is updated step by step. The only one subproblem in the proposed algorithm has a closed-form solution, and the extrapolation coefficients can be chosen to satisfy $\sup_k \beta_k=1$. After finitely many iterations, the support set of the iterate does not change. Further, we give sufficient conditions to guarantee that any accumulation point of $\{x^k\}$ generated by the proposed algorithm is a local minimizer of the considered problem. For a class of extrapolation coefficients, we obtain not only the convergence of the sequence, but also the convergence rate of $O(\ln k/k)$ on the loss and objective function values. Moreover, we consider in particular the case that the loss function is a Lipschitz continuous convex function. To solve it, we provide an algorithm, which can be viewed as a specific form of the above proposed algorithm. This algorithm owns both sequence convergence on the iterates and convergence rate of $o(k^{-2})$ on the loss and objective function values. Additionally, the limit point not only is a local minimizer of the considered problem but also possesses a desirable lower bound.
\begin{acknowledgements}
This work is funded by the National Science Foundation of China (Nos: 11871178,61773136).
\end{acknowledgements}

% BibTeX users please use one of
%\bibliographystyle{spbasic}      % basic style, author-year citations
\bibliographystyle{spmpsci}      % mathematics and physical sciences

\bibliography{reference}   % name your BibTeX data base

% Non-BibTeX users please use

%\begin{thebibliography}{}
%
% and use \bibitem to create references. Consult the Instructions
% for authors for reference list style.
%
%\bibitem{re1}
%\bibitem{RefJ}
% Format for Journal Reference
% Format for books
%\bibitem{RefB}
% etc
%\end{thebibliography}

\end{document}